\documentclass[12pt]{article}
\usepackage{a4wide}
\usepackage{amsmath,amssymb,amsthm}
\usepackage[mathscr]{eucal}
\usepackage{graphics}
\usepackage{epsfig}
\usepackage{psfrag}
\usepackage[usenames]{color}

\definecolor{red}{rgb}{1,0.00,0.00}
\author{Ievgen Bondarenko, Daniele D'Angeli, Tatiana Nagnibeda
\footnote{The second author was supported by Austrian Science Fund project FWF P24028-N18. The third author
acknowledge the support of the Swiss National Science Foundation Grant PP0022-118946. }}
\title{\textbf{Ends of Schreier graphs and cut-points of limit spaces of self-similar groups}}

\newcommand{\xmo}{X^{-\omega}}
\newcommand{\xo}{X^\omega}

\newcommand{\limGs}[1][]{\mathcal{X}_{#1}}
\newcommand{\lims}[1][]{\mathscr{J}_{#1}}
\newcommand{\si}{\mathsf{s}}

\newcommand{\muls}{\mathsf{m}}
\newcommand{\cc}{c}
\newcommand{\ic}{ic}
\newcommand{\pc}{pc}

\newcommand{\tile}{\mathscr{T}}
\newcommand{\nucl}{\mathcal{N}}
\newcommand{\pcset}{\mathscr{P}}

\newcommand{\gr}{\Gamma}
\newcommand{\B}{\mathsf{B}}

\newcommand{\A}{\mathsf{A}}

\newcommand{\tm}{t}
\newcommand{\om}{o}
\newcommand{\Rs}{\hat{\mathbb{C}}}

\newcommand{\M}{\mathcal{M}}

\newtheorem{theor}{Theorem}
\newtheorem{prop}[theor]{Proposition}
\newtheorem{cor}[theor]{Corollary}
\newtheorem{lemma}{Lemma}
\theoremstyle{definition}
\newtheorem{defi}{Definition}
\newtheorem{example}{Example}
\newtheorem{rem}{Remark}

\begin{document}
\maketitle

\begin{abstract}
Every self-similar group acts on the space $\xo$ of infinite words over some alphabet $X$. We study the
Schreier graphs $\gr_w$ for $w\in\xo$ of the action of self-similar groups generated by bounded automata
on the space $\xo$. Using sofic subshifts we determine the number of ends for every Schreier graph
$\gr_w$. Almost all Schreier graphs $\gr_w$ with respect to the uniform measure on $\xo$ have one or two
ends, and we characterize bounded automata whose Schreier graphs have two ends almost surely. The
connection with (local) cut-points of limit spaces of self-similar groups is established.\vspace{0.2cm}

\vspace{0.2cm}


\end{abstract}

%20F65 (Geometric group theory),
%05C63 (2010-now)  Infinite graphs
%05C25 (1973-now)  Graphs and abstract algebra
%20F69 (2000-now)  Asymptotic properties of groups
%20E08 (1991-now)  Groups acting on trees

\tableofcontents

\section{Introduction}

One of the fundamental properties of fractal objects is self-similarity, which means that pieces of an
object are similar to the whole object. In the last twenty years the notion of self-similarity has
successfully penetrated into algebra. This lead to the development of many interesting constructions such
as self-similar groups and semigroups, iterated monodromy groups, self-iterating Lie algebras,
permutational bimodules, etc. The first examples of self-similar groups showed that these groups enjoy
many fascinating properties (torsion, intermediate growth, finite width, just-infiniteness and many
others) and provide counterexamples to several open problems in group theory. Later, Nekrashevych showed
that self-similar groups appear naturally in dynamical systems as iterated monodromy groups of
self-coverings and provide combinatorial models for iterations of self-coverings.

Self-similar groups are defined by their action on the space $X^{*}$ of all finite words over a finite
alphabet $X$ --- one the most basic self-similar object. A faithful action of a group $G$ on $X^{*}$ is
called self-similar if for every $x\in X$ and $g\in G$ there exist $y\in X$ and $h\in G$ such that
$g(xv)=yh(v)$ for all words $v\in X^{*}$. The self-similarity of the action is reflected in the property
that the action of any group element on a piece $xX^{*}$ (all words with the first letter $x$) of the
space $X^{*}$ can be identified with the action of another group element on the whole space $X^{*}$. We
can also imagine the set $X^{*}$ as the vertex set of a regular rooted tree with edges $(v,vx)$ for $x\in
X$ and $v\in X^{*}$. Then every self-similar group acts by automorphisms on this tree. Alternatively,
self-similar groups can be defined as groups generated by the states of invertible Mealy automata which
are also known as automaton groups or groups generated by automata.  All these interpretations come from
different applications of self-similar groups in diverse areas of mathematics: geometric group theory,
holomorphic dynamics, fractal geometry, automata theory, etc. (see
\cite{self_sim_groups,fractal_gr_sets,GNS} and the references therein).

In this paper we consider Schreier graphs of self-similar actions of groups. Given a group $G$ generated by
a finite set $S$ and acting on a set $M$, one can associate to it the (simplicial) Schreier graph
$\Gamma(G,M, S)$: the vertex set of the graph is the set $M$, and two vertices $x$ and $y$ are adjacent if
and only if there exists $s\in S\cup S^{-1}$ such that $s(x)=y$. Schreier graphs are generalizations of the
Cayley graph of a group, which corresponds to the action of a group on itself by the multiplication from
the left.

Every self-similar group $G$ preserves the length of words in its action on the space $X^{*}$. We then
have a family of natural actions of $G$ on the sets $X^n$ of words of length $n$ over $X$. From these
actions one gets a family of corresponding finite Schreier graphs $\{\Gamma_n\}_{n\geq 1}$. It was noticed
in \cite{barth_gri:spectr_Hecke} that for a few self-similar groups the graphs $\{\gr_n\}_{n\geq 1}$ are
substitution graphs (see \cite{Previte}) --- they can be constructed by a finite collection of vertex
replacement rules
--- and, normalized to have diameter one, they converge in the Gromov-Hausdorff metric to certain fractal
spaces. However, in general, the Schreier graphs of self-similar groups are neither substitutional nor
self-similar in any studied way (see discussion in \cite[Section~I.4]{PhDBondarenko}). Nevertheless, the
observation from \cite{barth_gri:spectr_Hecke} lead to the notion of the limit space of a self-similar
group introduced by Nekrashevych \cite{self_sim_groups}, which usually have fractal structure. Although
finite Schreier graphs $\{\gr_n\}_{n\geq 1}$ of a group do not necessary converge to the limit space, they
form a sequence of combinatorial approximations to it.

Any self-similar action can be extended to the set $\xo$ of right-infinite words over $X$ (boundary of the
tree $X^{*}$). Therefore we can also consider the uncountable family of Schreier graphs
$\{\Gamma_w\}_{w\in\xo}$ corresponding to the action of the group on the orbit of $w$. Each Schreier graph
$\Gamma_w$ can be obtained as a limit of the sequence $\{\Gamma_n\}_{n\geq 1}$ in the space
$\mathcal{G}^{*}$ of (isomorphism classes of) pointed graphs with pointed Gromov-Hausdorff topology. The
map $\theta: X^{\omega}\rightarrow \mathcal{G}^{*}$ sending a point $w$ to the isomorphism class of the
pointed graph $(\Gamma_w,w)$ pushes forward the uniform probability measure on the space $X^{\omega}$ to a
probability measure on the space of Schreier graphs. This measure is the so-called Benjamini-Schramm limit
of the sequence of finite graphs $\{\Gamma_n\}_{n\geq 1}$. Therefore the family of Schreier graphs $\gr_w$
and the limit space represent two limiting constructions associated to the action and to the sequence of
finite Schreier graphs $\{\Gamma_n\}_{n\geq 1}$. Structure of these Schreier graphs as well as some of
their properties such as spectra, expansion, growth, random weak limits, probabilistic models on them,
have been studied in various works over the last ten years, see
\cite{barth_gri:spectr_Hecke,gri_zuk:lampl_group,gri_sunik:hanoi,PhDBondarenko,
ddmn:GraphsBasilica,ddn:Ising,ddn:Dimer,GrowthSch,MN_AMS,ZigZag:bond}.

Most of the studied self-similar groups are generated by the so-called bounded automata introduced by
Sidki in \cite{sidki:circ}. The structure of bounded automata is clearly understood, which allows one to
deal fairly easily with groups generated by such automata. The main property of bounded automaton groups
is that their action is concentrated along a finite number of ``directions'' in the tree $X^{*}$. Every
group generated by a bounded automaton belongs to an important class of contracting self-similar groups,
which appear naturally in the study of expanding (partial) self-coverings of topological spaces and
orbispaces, as their iterated monodromy groups \cite{self_sim_groups,img}. Moreover, all iterated
monodromy groups of post-criticaly finite polynomials are generated by bounded automata. The limit space
of an iterated monodromy group of an expanding (partial) self-covering $f$ is homeomorphic to the Julia
set of the map $f$.  In the language of limit spaces, groups generated by bounded automata are precisely
those finitely generated self-similar groups whose limit spaces are post-critically finite self-similar
sets (see \cite{bondnek:pcf}). Such sets play an important role in the development of analysis on fractals
(see \cite{kigami:anal_fract}).

The main goal of this paper is to investigate the ends of the Schreier graphs $\{\gr_w\}_{w\in\xo}$ of
self-similar groups generated by bounded automata and the corresponding limit spaces. The number of ends
is an important asymptotic invariant of an infinite graph.  Roughly speaking, each end represents a
topologically distinct way to move to infinity inside the graph.  The most convenient way to define an end
in an infinite graph $\Gamma$ is by the equivalence relation on infinite rays in $\Gamma$, where two rays
are declared equivalent if their tails lie in the same connected component of $\Gamma \setminus F$ for any
finite subgraph $F$ of $\Gamma$. Any equivalence class is an \emph{end} of the graph $\Gamma$. The number
of ends is a quasi-isometric invariant. The Cayley graph of an infinite finitely generated group can have
one, two or infinitely many ends. Two-ended groups are virtually infinite cyclic and the celebrated
theorem of Stallings characterizes finitely generated groups with infinite number of ends.

%%%%%%%%%%%%%%%%%%%%%%%%%%%%%%%%%%%%%%%%%%%%%%%

\subsection*{Plan of the paper and main results}
%%%%%%%%%%%%%%%%%%%%%%%%%%%%%%%%%%%%%%%%%%%%%%%%%%

Our main results can be summarized as follows.

\begin{itemize}
 \item Given a group generated by a bounded automaton we exhibit a constructive method that determines, for a given right-infinite word $w$, the number of ends of the Schreier graph $\Gamma_w$ (Section~\ref{Section_Ends}).

\item We show that the Schreier graphs $\Gamma_w$ of a group generated by a bounded automaton have either one or
two ends almost surely, and there are only finitely many Schreier graphs with more than two ends (Section~\ref{Section_number_ends}, Proposition~\ref{prop_more_ends}, Theorem~\ref{th_classification_two_ends}).

\item We classify bounded automata generating groups whose Schreier graphs $\Gamma_w$ have almost surely two ends (Theorem~\ref{th_classification_two_ends}). In the binary case, these groups agree with the class of groups defined by \v{S}uni\'{c} in
\cite{Zoran} (Theorem~\ref{th_Sch_2ended_binary}).

\item We exhibit a constructive method that describes cut-points of limit spaces of groups generated by
bounded automata (Section~\ref{section_component}).

\item In particular, we get a constructive method that describes cut-points of Julia sets of
post-critically finite polynomials (Section~\ref{section_component}).

\item We show that a punctured limit space has one or two connected components almost surely (Theorem~\ref{thm_punctured_tile_ends}).

\item We classify contracting self-similar groups whose limit space is homeomorphic to an interval or a circle (Corollary~\ref{cor_limsp_interval_circle}).

\end{itemize}

The paper is organized as follows. First, we determine the number of connected components in the Schreier
graph $\gr_n\setminus v$ with removed vertex $v$. The answer comes from a finite deterministic acceptor
automaton over the alphabet $X$ (Section~\ref{finite automaton section}) so that given a word
$v=x_1x_2\ldots x_n$ the automaton returns the number of components in $\gr_n\setminus v$
(Theorem~\ref{thm_number_of_con_comp}). Using this automaton we determine the number of finite and
infinite connected components in $\gr_w\setminus w$ for any $w\in\xo$. By establishing the connection
between the number of ends of the Schreier graph $\gr_w$ and the number of infinite components in
$\gr_w\setminus w$ (Proposition~\ref{prop_ends_limit}), we determine the number of ends of $\gr_w$
(Theorem~\ref{thm_number_of_ends}).

If a self-similar group acts transitively on $X^n$ for all $n\in\mathbb{N}$, the action on $\xo$ is
ergodic with respect to the uniform measure on $\xo$, and therefore the Schreier graphs $\gr_w$ for
$w\in\xo$ have the same number of ends almost surely. For a group
generated by a bounded automaton %this
the ``typical'' number of ends is one or two, and we show that in most cases it is one, by characterizing
completely the bounded automata generating groups whose Schreier graphs $\gr_w$ have almost surely two
ends (Theorem~\ref{th_classification_two_ends}). In the binary case we show that automata giving rise to
groups whose Schreier graphs have almost surely two ends correspond to the adding machine or to one of the
automata defined by \v{S}uni\'{c} in \cite{Zoran} (Theorem~\ref{th_Sch_2ended_binary}).

In Section~\ref{Section_Cut-points} we recall the notion of the limit space of a contracting self-similar
group, and study the number of connected components in a punctured limit space for groups generated by bounded automata. We show that the number of ends in a typical Schreier graph $\gr_w$ coincides with the number of
connected components in a typical punctured neighborhood (punctured tile) of the limit space (Theorem~\ref{thm_punctured_tile_ends}). In particular, this number is equal to one or two. This fact is well-known for connected Julia sets of polynomials. While Zdunik \cite{zdunik} and Smirnov \cite{smirnov} proved that almost every point of a connected polynomial Julia set is a bisection point only when the polynomial is conjugate to a Chebyshev polynomial, we describe bounded automata whose limit spaces have this property. Moreover, we provide a constructive method to compute the number of connected components in a punctured limit space (Section~\ref{cut-points
section}). Finally, using the results about ends of Schreier graphs, we classify contracting self-similar
groups whose limit space is homeomorphic to an interval or a circle
(Corollary~\ref{cor_limsp_interval_circle}). This result agrees with the description of automaton groups
whose limit dynamical system is conjugate to the tent map given by Nekrashevych and \v{S}uni\'{c} in
\cite{img}.

In Section~\ref{Section_Preliminaries} we recall all needed definitions concerning self-similar groups,
automata and their Schreier graphs. In Section 5 we illustrate our results by performing explicit
computations for three concrete examples: the Basilica group, the Gupta-Fabrykowski group, and the
iterated monodromy group of $z^2+i$.

\vspace{0.2cm}\noindent\textbf{Acknowledgments.} The substantial part of this work was done while the first author was visiting the Geneva University, whose support and hospitality are gratefully acknowledged.

\section{Preliminaries}\label{Section_Preliminaries}

In this section we review the basic definitions and facts concerning self-similar groups, bounded automata
and their Schreier graphs. For more detailed information and for further references,
see~\cite{self_sim_groups}.

\subsection{Self-similar groups and automata}

Let $X$ be a finite set with at least two elements. Denote by $X^{*}=\{x_1x_2\ldots x_n | x_i\in X, n\geq
0\}$ the set of all finite words over $X$ (including the empty word denoted $\emptyset$) and with $X^n$
the set of words of length $n$. The length of a word $v=x_1x_2\ldots x_n\in X^n$ is denoted by $|v| = n$.

We shall also consider the sets $\xo$ and $\xmo$ of all right-infinite sequences $x_1x_2\ldots$, $x_i\in
X$, and left-infinite sequences $\ldots x_2x_1$, $x_i\in X$, respectively with the product topology of
discrete sets $X$. For an infinite sequence $w=x_1x_2\ldots$ (or $w=\ldots x_2x_1$) we use notation
$w_n=x_1x_2\ldots x_n$ (respectively, $w_n=x_n\ldots x_2x_1$). For a word $v$ we use notations $v^{\omega}=vv\ldots$ and $v^{-\omega}=\ldots vv$. The \textit{uniform Bernoulli measure} on
each space $\xo$ and $\xmo$ is the product measure of uniform distributions on $X$. The \textit{shift
$\sigma$} on the space $\xo$ (respectively, on $\xmo$) is the map which deletes the first (respectively,
the last) letter of a right-infinite (respectively, left-infinite) sequence.

\subsubsection{Self-similar groups}

\begin{defi}
A faithful action of a group $G$ on the set $X^{*}\cup\xo$ is called $\emph{self-similar}$ if for every
$g\in G$ and $x\in X$ there exist $h\in G$ and $y\in X$ such that
$$
g(xw)=yh(w)
$$
for all $w\in X^{*}\cup \xo$. The element $h$ is called the \textit{restriction} of $g$ at $x$ and is
denoted by $h=g|_x$.
\end{defi}

Inductively one defines the restriction $g|_{x_1x_2\ldots x_n}=g|_{x_1}|_{x_2}\ldots
|_{x_n}$ for every word $x_1x_2\ldots x_n\in X^{*}$. Restrictions have the following properties
\[
g(vu)=g(v)g|_v(u),\qquad g|_{vu}=g|_{v}|_{u}, \qquad (g\cdot h)|_v=g|_{h(v)}\cdot h|_v
\]
for all $g,h\in G$ and $v,u\in X^{*}$ (we are using left actions so that $(g h)(v)=g(h(v))$). If $X=\{1,2,\ldots,d\}$ then every element $g\in G$ can be uniquely represented by the tuple $(g|_1,g|_2,\ldots,g|_d)\pi_g$, where $\pi_g$ is the permutation induced by $g$ on the set $X$.

It follows from the definition that every self-similar group $G$ preserves the length of words under
its action on the space $X^{*}$, so that we have an action of the group $G$ on the set $X^n$ for every $n$.

The set $X^{*}$ can be naturally identified with a rooted regular tree where the root is labeled by the
empty word $\emptyset$, the first level is labeled by the elements in $X$ and the $n$-th level corresponds
to $X^n$. The set $\xo$ can be identified with the boundary of the tree. Every
self-similar group acts by automorphisms on this rooted tree and by homeomorphisms on its boundary.

\subsubsection{Automata and automaton groups}

Another way to introduce self-similar groups is through input-output automata and automaton groups. A
transducer \textit{automaton} is a quadruple $(S,X,\tm,\om)$, where $S$ is the set of states of automaton;
$X$ is an alphabet; $\tm: S\times X \rightarrow S$ is the transition map; and $\om: S\times X \rightarrow
X$ is the output map. We will use notation $S$ for both the set of states and the automaton itself. An
automaton is \textit{finite} if it has finitely many states and it is \textit{invertible} if, for all
$s\in S$, the transformation $\om(s, \cdot):X\rightarrow X$ is a permutation of $X$. An automaton can be
represented by a directed labeled graph whose vertices are identified with the states and for every state
$s\in S$ and every letter $x\in X$ it has an arrow from $s$ to $\tm(s,x)$ labeled by $x|\om(s,x)$. This
graph contains complete information about the automaton and we will identify them. When talking about
paths and cycles in automata we always mean directed paths and cycles in corresponding graph
representations.

Every state $s\in S$ of an automaton defines a transformation on the set $X^{\ast}\cup X^{\omega}$, which is again denoted by $s$ by abuse of notation, as follows. Given a word $x_1x_2\ldots$ over $X$, there exists a unique path in $S$ starting at the state $s$ and labeled by $x_1|y_1$, $x_2|y_2$, \ldots for some $y_i\in X$. Then $s(x_1x_2\ldots)=y_1y_2\ldots$. We always assume that our automata are minimal, i.e., different states define different transformations. The state of automata that defines the identity transformation is denoted by $1$.

An automaton $S$ is invertible when all transformations defined by its states are invertible. In this case one can consider the group
generated by these transformations under composition of functions, which is called the \textit{automaton
group} generated by $S$ and is denoted by $G(S)$. The natural action of every automaton group on the
space $\xo$ is self-similar, and vise versa, every self-similar action of a group $G$ can be given by the
automaton with the set of states $G$ and arrows $g\rightarrow g|_x$ labeled by $x|g(x)$ for all $g\in G$
and $x\in X$.

%and vise versa, every self-similar group $G$ can be given (generated) by its complete automaton $\A(G)$ of the action.
%The states of $\A(G)$ are the elements of the group $G$, and there is an edge $g\rightarrow g|_x$ labeled by $x|g(x)$ for
%every $g\in G$ and $x\in X$.

\subsubsection{Contracting self-similar groups}
A self-similar group $G$ is called \textit{contracting} if there exists a finite set $\nucl\subset G$ with
the property that for every $g\in G$ there exists $n\in\mathbb{N}$ such that $g|_v\in\nucl$ for all words
$v$ of length greater or equal to $n$. The smallest set $\nucl$ with this property is called the
\textit{nucleus} of the group. It is clear from definition that $h|_x\in\nucl$ for every $h\in\nucl$ and
$x\in X$, and therefore the nucleus $\nucl$ can be considered as an automaton. Moreover, every state of
$\nucl$ has an incoming arrow, because otherwise minimality of the nucleus would be violated. Also, the
nucleus is symmetric, i.e., $h^{-1}\in\nucl$ for every $h\in\nucl$.

A self-similar group $G$ is called \textit{self-replicating} (or recurrent) if it acts transitively on $X$,
and the map $g\mapsto g|_x$ from the stabilizer $Stab_G(x)$ to the group $G$ is surjective for some (every)
letter $x\in X$. It can be shown that a self-replicating group acts transitively on $X^n$ for every $n\geq
1$. It is also easy to see (\cite[Proposition~2.11.3]{self_sim_groups}) that if a finitely generated
contracting group is self-replicating then its nucleus $\nucl$ is a generating set.

\subsection{Schreier graphs vs tile graphs of self-similar groups}\label{section schreier vs tile}

Let $G$ be a group generated by a finite set $S$ and let $H$ be a subgroup of $G$. The
\textit{(simplicial) Schreier coset graph $\gr(G,S,H)$} of the group $G$ is the graph whose vertices are the
left cosets $G/H=\{gH : g\in G\}$, and two vertices $g_1H$ and $g_2H$ are adjacent if there exists $s\in
S$ such that $g_2H=sg_1H$ or $g_1H=sg_2H$.

\begin{defi}
Let $G$ be a group acting on a set $M$, then the corresponding (simplicial) \textit{Schreier graph
$\gr(G,S,M)$} is the graph with the set of vertices $M$, and two vertices $v$ and $u$ are adjacent if
there exists $s\in S$ such that $s(v)=u$ or $s(u)=v$.
\end{defi}

If the action $(G,M)$ is transitive, then
the Schreier coset graph $\gr(G,S,M)$ is isomorphic to the Schreier graph $\gr(G,S,Stab_G(m))$ of the group with
respect to the stabilizer $Stab_G(m)$ for any $m\in M$.

Let $G$ be a self-similar group generated by a finite set $S$. The sets $X^n$ are invariant under the
action of $G$, and we denote the associated Schreier graphs by $\gr_n=\gr_n(G,S)$. For a point $w\in\xo$ we
consider the action of the group $G$ on the $G$-orbit of $w$, and the associated Schreier graph is called
the \textit{orbital Schreier graph} denoted $\gr_w=\gr_w(G,S)$. For every $w\in\xo$ we have
$Stab_G(w)=\bigcap_{n\geq 1} Stab_G(w_n)$, where $w_n$ denotes the prefix of length $n$ of the infinite
word $w$. The connected component of the rooted graph $(\Gamma_n,w_n)$ around the root $w_n$ is exactly the
Schreier graph of $G$ with respect to the stabilizer of $w_n$. It follows immediately that the graphs
$(\Gamma_n,w_n)$ converge to the graph $(\Gamma_{w},w)$ in the pointed Gromov-Hausdorff topology \cite{gromov}.

Besides the Schreier graphs, we will also work with their subgraphs called tile graphs.

\begin{defi}
The \textit{tile graph} $T_n=T_n(G,S)$ is the graph with the set of vertices $X^n$, where two vertices $v$
and $u$ are adjacent if there exists $s\in S$ such that $s(v)=u$ and $s|_v=1$.
\end{defi}
The tile graph $T_n$ is thus a subgraph of the Schreier graph $\gr_n$. To define a tile graph for the action on
the space $\xo$, consider the same set of vertices as in $\gr_w$ and connect vertices $v$ and $u$ by an
edge if there exists $s\in S$ such that $s(v)=u$ and $s|_{v'}=1$ for some finite beginning $v'\in X^{*}$ of
the sequence $v$. The connected component of this graph containing the vertex $w$ is called the
\textit{orbital tile graph}~$T_w$. It is clear from the construction that we also have the convergence
$(T_n,w_n)\rightarrow (T_{w},w)$ in the pointed Gromov-Hausdorff topology.

The study of orbital tile graphs $T_w$ is based on the approximation by finite tile graphs $T_n$. Namely,
we will frequently use the following observation. Every tile graph $T_n$ can be considered as a subgraph of
$T_w$ under the inclusion $v\mapsto v\sigma^n(w)$. Indeed, if $v$ and $u$ are adjacent in $T_n$ then
$v\sigma^n(w)$ and $u\sigma^n(w)$ are adjacent in $T_w$. Moreover, every edge of $T_w$ appears in the graph
$T_n$ for all large enough $n$. Hence the graphs $T_n$ viewed as subgraphs of $T_w$ form a cover of $T_w$.

\subsection{Schreier graphs of groups generated by bounded automata}\label{subsection_BoundedAutomata}

%In this paper we study the Schreier graphs of self-similar groups generated by bounded automata.
%%%%%%%%%%%%%%%%%%%%%%%%%%%%%%%%%%%%%%%%%%%%%%%%%%%%%%%%%%%%%%%%%%%%%%%%%%%%%%%%%%%%%%5

\subsubsection{Bounded automata}

%Let $1$ be the identity state of automata.

\begin{defi} (Sidki \cite{sidki:circ})
A finite invertible automaton $S$ is called \textit{bounded} if one of the following equivalent conditions holds:
\begin{enumerate}
  \item[1)] the number of paths of length $n$ in $S\setminus\{1\}$ is bounded independently on $n$;

  \item[2)] the number of left- (equivalently, right-) infinite
paths in $S\setminus\{1\}$ is finite;

  \item[3)] any two
nontrivial cycles in the automaton are disjoint and not connected by a path, where a cycle is called trivial if it is a loop at the trivial state;

  \item[4)] the number of left- (equivalently, right-) infinite sequences, which are read along left- (respectively, right-) infinite paths in $S\setminus\{1\}$, is finite.
\end{enumerate}

\end{defi}

The states of a bounded automaton $S$ can be classified as follows:\\
\indent -- \ a state $s$ is \textit{finitary} if there exists $n\in\mathbb{N}$ such that $s|_v=1$ for all $v\in X^n$;\\
\indent -- \ a state $s$ is \textit{circuit} if there exists a nonempty word $v\in X^n$ such that $s|_v=s$,\\
\indent \hspace{0.35cm} i.e., $s$ belongs to a cycle in $S$; in this case $s|_u$ is finitary for every $u\in X^n$, $u\neq
v$;\\
\indent -- \ for every state $s$ there exists $n\in\mathbb{N}$ such that for every $v\in X^n$ the state $s|_v$ is\\
\indent \hspace{0.35cm} either finitary or circuit.\\ By passing to a power $X^m$ of the alphabet $X$ every
bounded automaton can be brought to the \textit{basic form} (see
\cite[Proposition~3.9.11]{self_sim_groups}) in which the above items hold with $n=1$; in
particular, all cycles are loops, and $s|_x=1$ for every finitary state $s$ and every $x\in X$ (here for $m$
we can take an integer number which is greater than the diameter of the automaton and is a multiple of the
length of every simple cycle).

Every self-similar group $G$ generated by a bounded automaton is contracting (see
\cite[Theorem~3.9.12]{self_sim_groups}). Its nucleus is a bounded automaton, which contains only finitary
and circuit states (because every state of a nucleus should have an incoming arrow).

\subsubsection{Cofinality $\&$ post-critical, critical and regular sequences}

In this section we introduce the notion of critical and post-critical sequences that is fundamental for our analysis.

Let $G$ be a contracting self-similar group generated by an automaton $S$ and we assume $S=S^{-1}$. Let us describe the vertex sets of the orbital tile graphs $T_w=T_w(G,S)$. Two right- (or left-) infinite
sequences are called \textit{cofinal} if they differ only in finitely many letters. Cofinality is an
equivalence relation on $\xo$ and $\xmo$. The respective equivalence classes are called the cofinality
classes and they are denoted by $Cof(\cdot)$. The following statement characterizes vertices of tile
graphs in terms of the cofinal sequences.
\begin{lemma}\label{cofinality}
Suppose that the tile graphs $T_n(G,S)$ are connected for all $n\in\mathbb{N}$. Then for every $w\in\xo$
the cofinality class $Cof(w)$ is the set of vertices of the orbital tile graph $T_w(G,S)$.
\end{lemma}
\begin{proof}
If $g(v)=u$ for $v,u\in X^{\omega}$ and $g|_{v'}=1$ for a
finite beginning $v'$ of $v$ then $v$ and $u$ are cofinal. Conversely, since every graph $T_n$ is connected, for every $v,u\in X^{n}$ there exists $g\in G$ such that $g(v)=u$ and $g|_v=1$. Hence $g(vw)=uw$ for
all $w\in X^{\omega}$ and every two cofinal sequences belong to the same orbital tile graph.
\end{proof}

We classify infinite sequences over $X$ as follows. A left-infinite sequence $\ldots x_2x_1\in\xmo$ is
called \textit{post-critical} if there exists a left-infinite path $\ldots e_2e_1$ in the automaton
$S\setminus \{1\}$ labeled by $\ldots x_2x_1|\ldots y_2y_1$ for some $y_i\in X$. The set $\pcset$ of all
post-critical sequences is called \textit{post-critical}. A right-infinite sequence $w=x_1x_2\ldots\in\xo$
is called \textit{critical} if there exists a right-infinite path $e_1e_2\ldots$ in the automaton
$S\setminus \{1\}$ labeled by $x_1x_2\ldots|y_1y_2\ldots$ for some $y_i\in X$. It follows that every shift
$\sigma^n(w)$ of a critical sequence $w$ is again critical, and for every $n\in\mathbb{N}$ there exists
$v\in X^n$ such that $vw$ is critical (here we use the assumption that every path in $S$ can be continued
to the left). It is proved in \cite[Proposition~IV.18]{PhDBondarenko} (see also
\cite[Proposition~3.2.7]{self_sim_groups}) that the set of post-critical sequences coincides with the set
of sequences that can be read along left-infinite paths in the nucleus of the group with removed trivial
state. The same proof works for critical sequences. Therefore the sets of critical and post-critical
sequences do not depend on the chosen generating set (as soon as it satisfies assumption that every state of the automaton $S$ has an incoming arrow, and $S^{-1}=S$). Finally,
a sequence $w\in\xo$ is called \textit{regular} if the cofinality class of $w$ does not contain critical
sequences, or, equivalently, if the shifted sequence $\sigma^n(w)$ is not critical for every $n\geq 0$.
Notice that the cofinality class of a critical sequence contains sequences which are neither regular nor
critical.

\begin{prop}\label{prop_properties_of_sequences}
Suppose that the automaton $S$ is bounded. Then the sets of critical and post-critical sequences are finite. Every post-critical sequence is pre-periodic. Every cofinality class contains not more than one critical sequence. The cofinality class of a regular sequence contains only regular sequences. If $w$ is regular, then there exists a finite beginning $v$ of $w$ such that $s|_v=1$ for every $s\in S$.
\end{prop}
\begin{proof}
The number of right- and left-infinite paths avoiding the trivial state is finite in every bounded
automaton. Thus the number of critical and post-critical sequences is finite.

The pre-periodicity of post-critical sequences and the periodicity of critical sequences follow from the
cyclic structure of bounded automata. The statement about the cofinality class of a critical sequence
follows immediately, because different periodic sequences can not differ only in finitely many letters.

Finally, if $w=x_1x_2\ldots\in\xo$ is regular, then starting from any state $s\in S$ and following the
edges labeled by $x_1|*,x_2|*,\ldots$ we will end at the trivial state. Hence there exists $n$ such that
$s|_{x_1x_2\ldots x_n}=1$ for all $s\in S$.
\end{proof}

Note that if the automaton $S$ is in the basic form, then every post-critical sequence is of the form
$y^{-\omega}$ or $y^{-\omega}x$ for some letters $x,y\in X$, and every critical sequence is of the form
$x^{\omega}$ for some $x\in X$.

\subsubsection{Inflation of graphs}\label{section_inflation}

Let $G$ be a group generated by a bounded automaton $S$. We assume $S=S^{-1}$ and every state of $S$ has an incoming arrow. In what follows we describe an inductive method (called \textit{inflation of graphs}) to construct the tile graphs $T_n=T_n(G,S)$ developed in \cite[Chapter~V]{PhDBondarenko}.

Let $p=\ldots x_2x_1\in\pcset$ be a post-critical sequence. The vertex $p_n=x_n\ldots x_2x_1$ of the graphs $\gr_n$ and $T_n$ will be called \textit{post-critical}. Since the post-critical set $\pcset$ is finite, for all large enough $n$, the post-critical vertices of $\gr_n$ and $T_n$ are in  one-to-one correspondence with the elements of $\pcset$ (just take $n$ large enough so that $p_n\neq q_n$ whenever $p\neq q$). Hence, with a slight abuse of
notations, we will consider the elements of the set $\pcset$ as the vertices of the graphs $\gr_n$ and
$T_n$.

Let $E$ be the set of all pairs $\{(p,x),(q,y)\}$ for $p,q\in\pcset$ and $x,y\in X$ such that there exists
a left-infinite path in the automaton $S$, which ends in the trivial state and is labeled by the pair
$px|qy$.

\begin{theor}\cite[Theorem~V.8]{PhDBondarenko}\label{th_tile_graph_construction}
To construct the tile graph $T_{n+1}$  take $|X|$ copies of the tile graph $T_n$, identify their sets of
vertices with $X^nx$ for $x\in X$, and connect two vertices $vx$ and $uy$ by an edge if and only if
$v,u\in\pcset$ and $\{(v,x),(u,y)\}\in E$.
\end{theor}

The procedure of inflation of graphs given in Theorem~\ref{th_tile_graph_construction} can be described
using the graph $M$ with the vertex set $\pcset\times X$ and the edge set $E$, which we call the \textit{model graph} associated to the automaton $S$. The vertex $(p,x)$ of $M$ is called \textit{post-critical}, if the sequence $px$ is post-critical. Note that the post-critical vertices of $M$ are in one-to-one correspondence with elements of the post-critical set $\pcset$. Now if we ``place'' the
graph $T_n$ in the model graph instead of the vertices $\pcset\times x$ for each $x\in X$ such that the
post-critical vertices of $T_n$ fit with the set $\pcset\times x$, we get the graph $T_{n+1}$. Moreover,
the post-critical vertices of $M$ will correspond to the post-critical vertices of $T_{n+1}$.

In order to construct the Schreier graph $\gr_n$ we can take the tile graph $T_n$ and add an edge between
post-critical vertices $p$ and $q$ if $s(p)=q$ for some $s\in S$. Indeed, if $s(v)=u$ and $s|_v\neq 1$ (the
edge that does not appear in $T_n$) then $v$ and $u$ are post-critical vertices, and they should be
adjacent in $\gr_n$. Notice that there are only finitely many added edges (independently on $n$) and they
can be described directly through the generating set $S$. Namely, define the set $E(\gr\setminus T)$ as the
set of all pairs $\{p,q\}$ for $p,q\in\pcset$ such that there exists a left-infinite path in $S$ labeled by
the pair $p|q$ or $q|p$. Then if we take the tile graph $T_n$ and add an edge between $p_n$ and $q_n$ for
every $\{p,q\}\in E(\gr\setminus T)$, we get the Schreier graph $\gr_n$.

\section{Ends of tile graphs and Schreier graphs}\label{Section_Ends}

In this section we present the main results about the number of ends of Schreier graphs $\Gamma_w$ of groups
generated by bounded automata. Our method passes through the study of the same problem for the tile graphs $T_w$. First, we show that the number of ends of graphs $T_w$ can be deduced from the number of connected components in tile graphs with a vertex removed (see Proposition~\ref{prop_ends_limit}). We use the inflation procedure to construct a finite deterministic automaton $\A_{ic}$, which given a sequence $w\in\xo$ determines the number of infinite connected components in the graph $T_w$ with the vertex $w$ removed (see Proposition~\ref{prop_infcomp_A_ic}). Then we describe all sequences $w\in\xo$ such that $T_w$ has a given number of ends in terms of sofic subshifts associated to strongly connected components of the automaton $\A_{ic}$. Further we deduce the number of ends of Schreier graph $\Gamma_w$ (see Corollary \ref{remarktile}). After this, we pass to the study of the number of ends for a random Schreier graph $\Gamma_w$. We show that picking randomly an element $w$
in $X^{\omega}$, the graph $\Gamma_w$ has one or two ends (see Corollary~\ref{cor_one_or_two_ends_a_s}). The latter case is completely described (see Theorem~\ref{th_classification_two_ends}).

%The crucial observation is that, passing from one level to the next one, the
%tile (and Schreier) graphs (with a vertex removed) only ``grow'' towards the components containing the
%post-critical vertices. This property together with inflation procedure allows us to construct a finite deterministic
%automaton $\A_{ic}$ which given a sequence $w\in\xo$ returns the number of infinite connected components in the tile graph $T_w$ with a vertex $w$ removed. Clearly, this number is in general smaller than the number of ends, however, by using sofic subshift, we can describe a method that gives exactly the number of ends.

\subsection{Technical assumptions}

In what follows, except for a few special cases directly indicated, we make the following assumptions about
the studied self-similar groups $G$ and their generating sets $S$:\\
\\
\indent \hspace{0.3cm}  1. {\it The group $G$ is generated by a bounded automaton $S$.}\\
\\
\indent \hspace{0.3cm}  2. {\it The tile graphs $T_n=T_n(G,S)$ are connected.}\\
\\
\indent \hspace{0.3cm}  3. {\it Every state of the automaton $S$ has an incoming arrow, and $S^{-1}=S$.}\\
\\
%\indent \hspace{1cm}  2. The group $G$ is self-replicating.\\
Instead of the assumption 2 it is enough to require that the group acts transitively on $X^n$ for every
$n\geq 1$, i.e., the Schreier graphs $\gr_n(G,S)$ are connected. Then, even if the tile graphs $T_n(G,S)$
are not connected, there is a uniform bound on the number of connected components in $T_n(G,S)$ (see how
the Schreier graphs are constructed from the tile graphs after Theorem~\ref{th_tile_graph_construction}), and one can apply the developed methods to each component. The assumption 3 is technical, it
guaranties that every directed path in the automaton $S$ can be continued to the left. If the generating
set $S$ contains a state $s'$, which does not contain incoming edges, then $s|_x\in S\setminus\{s'\}$ for
every $s\in S$ and $x\in X$, and hence the state $s'$ does not interplay on the asymptotic properties of
the tile or Schreier graphs. Moreover, if the group is self-replicating, then the property 3 is always
satisfied, when we take its nucleus $\nucl$ as the generating set $S$.

%Under these assumptions the tile graphs $T_n=T_n(G,S)$ and the Schreier graphs $\gr_n=\gr_n(G,S)$ are
%connected.

\subsection{The number of ends and infinite components in tile graphs with a vertex removed}

For a graph $\gr$ and its vertex $v$ we denote by $\gr\setminus v$ the graph obtained from $\gr$ by
removing the vertex $v$ together with all edges adjacent to $v$.

Let $\Gamma$ be an infinite, locally finite graph. A \textit{ray} in $\Gamma$ is
an infinite sequence of adjacent vertices $v_1, v_2, \ldots $ in $\Gamma$ such that
$v_i\neq v_j$ for $i\neq j$. Two rays $r$ and $r'$ are equivalent if for every finite
subset $F\subset\Gamma$ infinitely many vertices of $r$ and $r'$ belong to the same
connected component of $\Gamma\setminus F$. An \textit{end} of $\Gamma$ is an
equivalence class of rays.

In what follows we use the notation:
\begin{itemize}
\item $\#Ends(\gr)$ is the number of ends of $\gr$;
\item $\cc(\gr)$ is the number of connected components in the graph $\gr$;
\item $\ic(\gr)$ is the number of infinite connected components in $\gr$.
\end{itemize}

We will show later that the number $\ic(\gr)$ can be computed for $\Gamma=T_w\setminus w$. The following proposition relates this value to the number of ends of $T_w$.

\begin{prop}\label{prop_ends_limit}
Every tile graph $T_w$ for $w\in\xo$ has finitely many ends, which is equal to
\[
\#Ends(T_w)=\lim_{n\rightarrow\infty} \ic(T_{\sigma^n(w)}\setminus \sigma^n(w)),
\]
where $\sigma$ is the shift map on the space $\xo$.
\end{prop}
\begin{proof}
Let us show that the number of infinite connected components of the graphs $T_{\sigma^n(w)}\setminus
\sigma^n(w)$ and $T_w\setminus X^n\sigma^n(w)$ is the same for every $n$. Consider the natural partition of
the set of vertices of $T_w$ given by
\[
Cof(w)=\bigsqcup_{w'\in Cof(\sigma^n(w))} X^nw'.
\]
Using the graph $T_w$ we construct a new graph $\mathscr{G}$ with the set of vertices $Cof(\sigma^n(w))$,
where two vertices $v$ and $u$ are adjacent if there exist $v',u'\in X^n$ such that $v'v$ and $u'u$ are
adjacent in $T_w$. The graph $\mathscr{G}$ is isomorphic to the tile graph $T_{\sigma^n(w)}$ under the
identity map on $Cof(\sigma^n(w))$. Indeed, let $v$ and $u$ be adjacent in $\mathscr{G}$. Then there exist
$v',u'\in X^n$ and $s\in S$ such that $s(v'v)=u'u$ and $s|_{v'v''}=1$ for a finite beginning $v''$ of $v$.
It follows that $s|_{v'}(v)=u$, $s|_{v'}|_{v''}=1$, and $s|_{v'}\in S$, because of self-similarity.
Therefore $v$ and $u$ are adjacent in $T_{\sigma^n(w)}$. Conversely, suppose $s(v)=u$ and $s|_{v''}=1$ for
some $s\in S$ and a finite beginning $v''$ of $v$. Since each element of $S$ has an incoming edge, there
exist $s'\in S$ and $v',u'\in X^n$ such that $s'(v'v)=u'u$ and $s'|_{v'v''}=1$. Hence $v$ and $u$ are
adjacent in the graph~$\mathscr{G}$.

The subgraph of $T_{w}$ spanned by every set of vertices $X^nw'$ for $w'\in Cof(\sigma^n(w))$ is connected,
because, by assumption, the tile graphs $T_n$ are connected. Hence, the number of infinite connected components in
$T_w\setminus X^n\sigma^n(w)$ is equal to the number of infinite connected components in
$T_{\sigma^n(w)}\setminus \sigma^n(w)$. In particular, this number is bounded by the size of the generating
set $S$.

Every infinite component of $T_w\setminus X^n \sigma^n(w)$ contains at least one end. Hence the estimate
\[
\#Ends(T_w)\geq \ic(T_w\setminus X^n \sigma^n(w))=\ic(T_{\sigma^n(w)}\setminus \sigma^n(w))
\]
holds for all $n$. In particular
\[
\#Ends(T_w)\geq \lim_{n\rightarrow\infty} \ic(T_{\sigma^n(w)}\setminus \sigma^n(w))
\]

For the converse consider the ends $\gamma_1,\ldots,\gamma_k$ of the graph $T_{w}$. They
can be made disconnected by removing finitely many vertices. Take $n$ large enough so that the set
$X^n\sigma^n(w)$ disconnects the ends $\gamma_i$. Since every end belongs to an infinite component, we get
at least $k$ infinite components of $T_{\sigma^n(w)}\setminus \sigma^n(w)$. In particular, the number of
ends is finite and the statement follows.
\end{proof}

In particular, the number of ends of every tile graph $T_w$ is not greater than the maximal degree of
vertices, i.e., $\# Ends(T_w)\leq |S|$. Now let us show how to compute the number $\ic(T_w\setminus w)$ in terms of the components that contain post-critical vertices. In other terms, only components of $T_n\setminus w_n$ with post-critical vertices give a positive contribution, in the limit, to the number of infinite components. In order to do that, we denote by $\pc(T_n\setminus w_n)$ the number of connected components of $T_n\setminus w_n$ that contain a
post-critical vertex. %If we drop the condition about infinity we get the following

\begin{prop}\label{prop_infcomp_limit}
Let $w=x_1x_2\ldots\in\xo$ be a regular or a critical sequence, then $\pc(T_n\setminus w_n)$ is an eventually non-increasing sequence and
\begin{eqnarray*}
\ic(T_w\setminus w)=\lim_{n\rightarrow\infty} \pc(T_n\setminus w_n).
\end{eqnarray*}
\end{prop}
\begin{proof}
Choose $n$ large enough so that the subgraph $T_n$ of $T_w$ contains all edges of $T_w$ adjacent to the
vertex $w$. Notice that if a vertex $v$ of $T_n$ is adjacent to some vertex $s(v)$ in $T_w\setminus T_n$,
then $s|_v\neq 1$ and thus $v$ is post-critical. It follows that if $C$ is a connected component of
$T_n\setminus w_n$ without post-critical vertices, then all the edges of the graph $T_w\setminus w$
adjacent to the component $C$ are contained in the graph $T_n\setminus w_n$. Hence $C$ is a finite
component of $T_w\setminus w$. Therefore the number of infinite components of $T_w\setminus w$ is not
greater than the number of components of $T_n\setminus w_n$ that contain a post-critical vertex. It
follows $\ic(T_w\setminus w)\leq \pc(T_{n+k}\setminus w_{n+k})\leq  \pc(T_n\setminus w_n)$ for all $k\geq
1$. In fact, if $v, v'$ are post-critical and belong to the same connected component of $T_n\setminus
w_n$, then there exists a path $\{s_1,\ldots, s_m\}$ connecting them such that $s_i|_{s_1\cdots
s_{i-1}(v_n)}=1$, in particular the same holds for $s_i|_{s_1\cdots s_{i-1}(v_{n+k})}=1$, $k\geq 1$. This
implies the monotonicity of $\pc(T_{n+k}\setminus w_{n+k})$ for $k\geq 1$.

Suppose now that $w$ is regular. Let $C$ be a finite component of $T_w\setminus w$. Since $C$ is finite,
every edge inside $C$ appears in the graph $T_n$ for all large enough $n$, and thus $C$ is a connected
component of $T_n\setminus w_n$. Since $C$ contains only regular sequences, the last statement in
Proposition~\ref{prop_properties_of_sequences} implies that for all $s\in S$ and every vertex $v$ in $C$
we have $s|_{v_n}=1$ for all large enough $n$. In other words, the vertex $v_n$ of $T_n$ is not
post-critical for every vertex $v$ in $C$. Therefore the component $C$ is not counted in the number
$\pc(T_n\setminus w_n)$. Hence $\ic(T_w\setminus w)=\pc(T_n\setminus w_n)$ for all large enough~$n$.

The same arguments work if $w$ is critical, because every cofinality class contains not more than one
critical sequence, and hence the graph $T_w\setminus w$ has no critical sequences.
\end{proof}

\begin{rem}
With a slight modification the last proposition also works for a sequence $w$, which is not critical but is
cofinal to some critical sequence $u$. In this case, we can count the number of connected components of
$T_n\setminus w_n$ that contain post-critical vertices other than $u_n$, and then pass to the limit to get
the number of infinite components in $T_w\setminus w$. Indeed, it is enough to notice that if the graph
$T_n\setminus w_n$ contains a connected component $C$ with precisely one post-critical vertex $u_n$ for
large enough $n$, then $C$ is a finite component in the graph $T_w\setminus w$. Under this modification the
proposition may be applied to any sequence.

Also to find the number of ends it is enough to know that the limit in Proposition~\ref{prop_infcomp_limit}
is valid for regular and critical sequences. For any sequence $w$ cofinal to a critical sequence $u$ we
just consider the graph $T_w=T_u$ centered at the vertex $u$ and apply the proposition.
\end{rem}

\begin{rem}
It is not difficult to observe that one can use the same method to obtain the number $\cc(T_w\setminus w)$
of all connected components of $T_w\setminus w$. In particular, one has
$$
\cc(T_w\setminus w)=\lim_{n\rightarrow\infty} \cc(T_n\setminus w_n).
$$
\end{rem}

\subsection{Finite automaton to determine the number of components in tile graphs with a
vertex removed}\label{finite automaton section}

Using the iterative construction of tile graphs given in Theorem~\ref{th_tile_graph_construction} we can
provide a recursive procedure to compute the numbers $\pc(T_n\setminus w_n)$. We will construct a finite
deterministic (acceptor) automaton $\A_{ic}$ with the following structure: it has a unique initial state,
each arrow in $\A_{ic}$ is labeled by a letter $x\in X$, each state of $\A_{ic}$ is labeled by a partition
of a subset of the post-critical set $\pcset$. The automaton $\A_{ic}$ will have the property that, given
a word $v\in X^n$, the final state of $\A_{ic}$ after reading $v$ corresponds to the partition of the
post-critical vertices of the graph $T_n$ induced by the connected components of $T_n\setminus v$. Then
$\pc(T_n\setminus v)$ is just the number of parts in this partition.

%, which we identify with the post-critical set $\pcset$

We start with the following crucial consideration for the construction of the automaton $\A_{ic}$. Let $v$
be a vertex of the tile graph $T_n$. The components of $T_n\setminus v$ partition the set of post-critical
vertices of $T_n$. Let us consider only those components that contain at least one post-critical vertex.
Let $\pcset_i\subset \pcset$ be the set of all post-critical sequences, which represent post-critical
vertices in $i$-th component. If the vertex $v$ is not post-critical, then $\sqcup_{i} \pcset_i=\pcset$.
Otherwise, $\sqcup_{i} \pcset_i$ is a proper subset of $\pcset$; every sequence $p$ in $\pcset\setminus
\sqcup_{i} \pcset_i$ represents $v$, i.e., $v=p_n$ (for all large enough $n$ the set $\pcset\setminus
\sqcup_{i} \pcset_i$ consists of just one post-critical sequence, while for small values of $n$ the same
vertex may be represented by several post-critical sequences). In any case, we say that $\{ \pcset_i \}_i$
is the \textit{partition} (of a subset of $\pcset$) \textit{induced by the vertex} $v$. If $T_n\setminus
v$ does not contain post-critical vertices (this happens when $v=p_n$ for every $p\in\pcset$), then we say
that $v$ induce the empty partition $\{\emptyset\}$.

The set of all partitions induced by the vertices of tile graphs is denoted by $\Pi$. The set $\Pi$ can be
computed algorithmically. To see this, let us show how, given the partition $P=\{ \pcset_i \}_i$ induced
by a vertex $v$ and a letter $x\in X$, one can find the partition $F=\{\mathscr{F}_j\}_j$ induced by the
vertex $vx$. We will use the model graph $M$ associated to the automaton $S$ in Section~\ref{section_inflation}, which has the vertex set $\pcset\times X$ and edges $E$. Recall that the set $\pcset$ is identified with the set of post-critical vertices of $M$. Let us construct the auxiliary graph $M_{P,x}$ as follows: take the model graph $M$, add
an edge between $(p,x)$ and $(q,x)$ for $p,q\in\pcset_i$ and every $i$, and add an edge between $(p,y)$
and $(q,y)$ for every $p,q\in\pcset$ and $y\in X, y\neq x$. Put $K=\{ (p,x) :
p\in\pcset\setminus\sqcup_{i} \pcset_i \}$. If the graph $M_{P,x}\setminus K$ contains no post-critical
vertices, then we define $\{\mathscr{F}_j\}_j$ as the empty partition $\{\emptyset\}$. Otherwise, we
consider the components of $M_{P,x}\setminus K$ with at least one post-critical vertex, and let
$\mathscr{F}_j\subset\pcset$ be the set of all post-critical vertices/sequences in $j$-th component.

\begin{lemma}\label{lemma construction Aic}
$\{\mathscr{F}_j\}_j$ is exactly the partition induced by the vertex $vx$.
\end{lemma}
\begin{proof}
Let us consider the map $\varphi:M_{P,x}\rightarrow T_{n+1}$ given by $\varphi((p,y))=p_ny$, $p\in\pcset$ and $y\in X$. This map is neither surjective, nor injective in general, nor a graph homomorphism. However, it preserves the inflation construction of the graph $T_{n+1}$ from the graph $T_n$. Namely, $\varphi$ maps each subset $\pcset\times\{y\}$ into the subset $X^ny$, and the edges in $E$ onto the edges of $T_{n+1}$ obtained under construction (see Theorem~\ref{th_tile_graph_construction}). Also, by definition of post-critical vertices, the map $\varphi$ sends the post-critical vertices of $M_{P,x}$ onto the post-critical vertices of $T_{n+1}$.

Note that the set $K$ is exactly the preimage of the vertex $vx$ under $\varphi$. Therefore we can consider the restriction $\varphi:M_{P,x}\setminus K\rightarrow T_{n+1}\setminus vx$.

For every $y\in X$, $y\neq x$ all vertices in $\pcset\times\{y\}$ belong to the same component of $M_{P,x}\setminus K$, and $\varphi$ maps these vertices to the same component of $T_{n+1}\setminus vx$, because the subgraph of $T_{n+1}$ induced by the set of vertices $X^ny$ contains the graph $T_n$, which is connected. For each $i$ the vertices in $\pcset_i\times\{x\}$ are mapped to the same component of the graph $T_{n+1}\setminus vx$, because its subgraph induced by the vertices $X^nx\setminus vx$ contains the graph $T_{n}\setminus v$ and $\pcset_i$ is its connected component.
It follows that, if two post-critical vertices $p$ and $q$ can be connected by a path $p=v_0,v_1,\ldots,v_m=q$ in $M_{P,x}\setminus K$, then for every $i$ the vertices $\varphi(v_i)$ and $\varphi(v_{i+1})$ belong to the same component of $T_{n+1}\setminus vx$, and therefore the post-critical vertices $\varphi(p)$ and $\varphi(q)$ lie in the same component.  Conversely, suppose $\varphi(p)$ and $\varphi(q)$ can be connected by a path $\gamma$ in $T_{n+1}\setminus vx$. We can subdivide $\gamma$ as $\gamma_1e_1\gamma_2e_2\ldots e_m\gamma_{m+1}$, where $e_i\in \varphi(E)$ and each subpath $\gamma_i$ is a path in a copy of $T_n$ inside $T_{n+1}$. The preimages of the end points of each $e_i$ belong to the same component in $M_{P,x}\setminus K$. Therefore $p$ and $q$ lie in the same component in $M_{P,x}\setminus K$.
The statement follows.
\end{proof}

%Every edge of $M_{P,x}$ between two vertices that lie in different subsets $\pcset\times\{y\}$ for $y\in X$ comes from the model graph $M$, i.e., it is an edge from $E$. Theorem~\ref{th_tile_graph_construction} implies that $\varphi$ maps these edges to edges of $T_{n+1}$, and moreover, $\varphi(E)$ is the set of all edges of $T_{n+1}$ between vertices that lie in different subsets $X^ny$ for $y\in X$.

It follows that we can find the set $\Pi$ algorithmically as follows. Note that the empty partition
$\{\emptyset\}$ is always an element of $\Pi$ (it is induced by the unique vertex of the tile $T_0$ of
zero level). We start with $\{\emptyset\}$ and for each letter $x\in X$ construct new partition
$\{\mathscr{F}_j\}_j$ as above. We repeat this process for each new partition until no new partition is
obtained. Since the set $\pcset$ is finite, the process stops in finite time. Then $\Pi$ is exactly the
set of all obtained partitions.

We construct the (acceptor) automaton $\A_{ic}$ over the alphabet $X$ on the set of states $\Pi$ with the
the unique initial state $\{\emptyset\}\in\Pi$. The transition function is given by the rule: for
$\{\pcset_i\}_i\in\Pi$ and $x\in X$ we put $\{\pcset_i\}_i \xrightarrow{\ x\ } \{\mathscr{F}_j\}_j$, where
$\{\mathscr{F}_j\}_j$ is defined as above. The automaton $\A_{ic}$ has all the properties we described at
the beginning of this subsection: given a word $v\in X^{*}$, the final state of $\A_{ic}$ after accepting
$v$ is exactly the partition induced by the vertex $v$. Since we are interested in the number
$\pc(T_n\setminus v)$ of components in $T_n\setminus v$ containing post-critical vertices, we can label
every state of $\A_{ic}$ by the number of components in the corresponding partition. We get the following
statement.

%The state $\bigcirc$ only takes into account the starting point of our computation, which is well defined
%for $n$ sufficiently large. More precisely, if $n$ is large enough to distinguish the post-critical
%sequences, we put $\phi(\bigcirc, v_n)= \bigsqcup_{i=1}^k \pcset_i$ if this is the admissible partition of
%post-critical vertices in $T_n\setminus v_n$.

%Given a word $v=x_1x_2\ldots x_k$, $k\geq n$, if we start at the initial state of the automaton $\A_{ic}$
%and follow the arrows labeled by $x_1,x_2,\ldots,x_k$ then we end at the state that correspond to the
%partition of the post-critical set (or of its subset) given by $T_k\setminus v$.

\begin{theor}\label{thm_number_of_con_comp}
The graph $T_{|v|}\setminus v$ has $k$ components containing a post-critical vertex if and only if the
final state of the automaton $\A_{ic}$ after accepting the word $v$ is labeled by the number $k$. In
particular, for every $k$ the set $C(k)$ of all words $v\in X^{*}$ such that the graph $T_{|v|}\setminus
v$ has $k$ connected components containing a post-critical vertex is a regular language recognized by the
automaton $\A_{ic}$.
\end{theor}

Similarly, we construct a finite deterministic acceptor automaton $\A_c$ for computing the number of all
components in tile graphs with a vertex removed. The states of the automaton $\A_c$ will be pairs of the form
$(\{\pcset_i\}_i,m)$, where $\{\pcset_i\}_i\in\Pi$ and $m$ is a non-negative integer number, which will
count the number of connected components without post-critical vertices. We start with the state
$(\{\emptyset\},0)$, which is the unique initial state of $\A_c$, and consequently construct new states
and arrows as follows. Let $(P=\{\pcset_i\}_i,m)$ be a state already constructed. For each $x\in X$ we
take the graph $M_{P,x}\setminus K$ constructed above, define $\{\mathscr{F}_j\}_j$ as above, and put
$m_x$ to be equal to the number of connected components in $M_{P,x}\setminus K$ without post-critical
vertices. If $(\{\mathscr{F}_j\}_j,m+m_x)$ is already a state of $\A_{c}$, then we put an arrow labeled by $x$ from the state $(\{\pcset_i\}_i,m)$ to the state $(\{\mathscr{F}_j\}_j,m+m_x)$.
Otherwise, we introduce $(\{\mathscr{F}_j\}_j,m+m_x)$ as a new state and put this arrow. We repeat this process for each new state
until no new state is obtained. Since the number of all components in $T_n\setminus v$ is not greater than
$|S|$, the number $m$ cannot exceed $|S|$, and the construction stops in finite time.

The automaton $\A_c$ has the following property: given a word $v\in X^{*}$, the final state of $\A_c$
after accepting the word $v$ is exactly the pair $(\{\pcset_i\}_i,m)$, where $\{\pcset_i\}_i$ is the
partition induced by $v$ and $m$ is the number of components in $T_{|v|}\setminus v$ without post-critical
vertices. Since we are interested only in the number $c(T_n\setminus v)$ of all components of
$T_n\setminus v$, we label every state $(\{\pcset_i\}_i,m)$ of the automaton $\A_{c}$ by the number $k+m$,
where $k$ is the number of sets in the partition $\{\pcset_i\}_i$. We get the following
statement.

\begin{prop}\label{prop_number_of_con_comp A_c}
The graph $T_{|v|}\setminus v$ has $k$ components if and only if the final state of the automaton $\A_{c}$
after accepting the word $v$ is labeled by the number $k$. In particular, for every $k$ the set $C(k)$ of
all words $v\in X^{*}$ such that the graph $T_{|v|}\setminus v$ has $k$ connected components is a regular
language recognized by the automaton $\A_{c}$.
\end{prop}

We need the following properties of vertex labels in the automata $\A_{ic}$ and $\A_{c}$.

\begin{lemma}\label{lemma_properties_A_ic}
\begin{enumerate}
  \item In every strongly connected component of the automata $\A_{ic}$ and $\A_c$ all states are labeled
  by the same number.

 \item All strongly connected components of the automaton $\A_{ic}$ without outgoing arrows are labeled by
 the same number.
\end{enumerate}
\end{lemma}
\begin{proof}
\textit{1.} Suppose that there is a strongly connected component with two states labeled by different
numbers. It would imply that there exists an infinite word such that the corresponding path in the
automaton passes through each of these states infinite number of times. We get a contradiction with Proposition~\ref{prop_infcomp_limit}, because the sequences $\pc(T_n\setminus w_n)$ and
$\cc(T_n\setminus w_n)$ are eventually monotonic (for the last one the proof is the same).

\textit{2.} Suppose there are two strongly connected components in the automaton $\A_{ic}$ without
outgoing arrows which are labeled by different numbers. Let $v$ and $u$ be finite words such that starting at the initial state of $\A_{ic}$ we end at the first and the second components respectively. Then for the infinite sequence $vuvu\ldots$ the
limit in Proposition~\ref{prop_ends_limit} does not exist, and we get a contradiction.
\end{proof}

Proposition~\ref{prop_infcomp_limit} together with Theorem~\ref{thm_number_of_con_comp} imply the
following method to find the number of infinite components in $T_w\setminus w$ for $w\in\xo$.

\begin{prop}\label{prop_infcomp_A_ic}
Let $w=x_1x_2\ldots\in\xo$ be a regular or critical sequence. The number of infinite connected components in $T_w\setminus w$ is equal to the label of a strongly connected component of $\A_{ic}$ that is visited infinitely often, when the automaton reads the sequence $w$.
\end{prop}

\subsection{The number of ends of tile graphs}

The characterization of the number of infinite components in the graph $T_w\setminus w$ together with
Proposition~\ref{prop_ends_limit} allows us to describe the number of ends of $T_w$.

Since every critical sequence $w$ is periodic, we can algorithmically find the number of ends of the graph
$T_w$ using Proposition~\ref{prop_ends_limit} and the automaton $\A_{ic}$. Now we introduce some
notations. Fix $k\geq 1$ and let $EC_{=k}$ be the union of cofinality classes of critical sequences $w$
whose tile graph $T_w$ has $k$ ends. Similarly we define the sets $EC_{>k}$ and $EC_{<k}$. Let
\begin{itemize}
\item $\A_{ic}(k)$ be the subgraph of $\A_{ic}$ spanned by the strongly connected components labeled by numbers $\geq k$;
\item $\mathscr{R}_{\geq k}$ be the one-sided sofic subshift given by the graph $\A_{ic}(k)$, i.e., $\mathscr{R}_{\geq k}$ is the set of all sequences that can be read along right-infinite paths in $\A_{ic}(k)$ starting at
any state;
\item $E_{\geq k}$ be the set of all sequences which are cofinal to some sequence from $\mathscr{R}_{\geq k}$. Since the set $\mathscr{R}_{\geq k}$ is shift-invariant, the set $E_{\geq k}$
coincides with $X^{*}\mathscr{R}_{\geq k}=\{vw | v\in X^{*}, w\in \mathscr{R}_{\geq k}\}$.
\end{itemize}

%Let $\A_{ic}(k)$ be the subgraph of $\A_{ic}$ spanned by the strongly connected components labeled by %%%%%%%%partitions of cardinality greater or equal to $k$
%the numbers $\geq k$. Consider the right sofic subshift $\mathscr{R}_{\geq k}$ given by the graph $\A_{ic}(k)$,
%which is the set of all sequences that can be read along right-infinite paths in $\A_{ic}(k)$ starting at
%any state. Define the set $E_{\geq k}$ of all sequences which are cofinal to some sequence from
%$\mathscr{R}_{\geq k}$. Since the set $\mathscr{R}_{\geq k}$ is shift-invariant, the set $E_{\geq k}$
%coincides with $X^{*}\mathscr{R}_{\geq k}=\{vw | v\in X^{*}, w\in \mathscr{R}_{\geq k}\}$.

\begin{theor}\label{thm_number_of_ends}
The tile graph $T_w$ has $\geq k$ ends if and only if $w\in E_{\geq k}\setminus EC_{<k}$. Hence, the tile
graph $T_w$ has $k$ ends if and only if $w\in E_{\geq k} \setminus \left(E_{\geq k+1} \cup EC_{>k}\cup
EC_{<k}\right)$.
\end{theor}
\begin{proof}
We need to prove that for a regular sequence $w$ the graph $T_w$ has at least $k$ ends if and only if
$w\in E_{\geq k}$. First, suppose $w\in E_{\geq k}$. Then $w$ is cofinal to a sequence $w'\in
\mathscr{R}_{\geq k}$, which is also regular. The sequences $w$ and $w'$ belong to the same tile graph
$T_w=T_{w'}$. There exists a finite word $v$ such that for the sequence $vw'$ the corresponding path in
$\A_{ic}$ starting at the initial state eventually lies in the subgraph $\A_{ic}(k)$. Then the graph
$T_{vw'}\setminus vw'$ has $\geq k$ infinite components by Proposition~\ref{prop_infcomp_limit}. Using the
correspondence between infinite components of $T_{w'}\setminus w'$ and of $T_{vw'}\setminus X^{|v|}w'$
shown in the proof of Proposition~\ref{prop_ends_limit} we get that the graph $T_{w'}\setminus w'$ has
$\geq k$ infinite components, and hence $T_{w'}=T_w$ has $\geq k$ ends.

For the converse, suppose the graph $T_{w}$ has $\geq k$ ends and the sequence $w$ is regular. Then
$ic(T_{\sigma^n(w)}\setminus \sigma^n(w))\geq k$ for some $n$ by Propositions~\ref{prop_ends_limit} and
\ref{prop_infcomp_limit}. Hence some shift $\sigma^n(w)$ of the sequence $w$ is in $\mathscr{R}_{\geq k}$
and thus $w\in E_{\geq k}$.
\end{proof}

Example with $IMG(z^2+i)$ in Section~\ref{Section_Examples} shows that we cannot expect to get a
description using subshifts of finite type, and indeed the description using sofic subshifts is the best
possible in these settings.

\subsection{From ends of tile graphs to ends of Schreier graphs}

Now we can describe how to derive the number of ends of Schreier graphs from the number of ends of tile
graphs.

\begin{prop}\label{prop_Schreier_construction}
\begin{enumerate}
\item The Schreier graph $\gr_w$ coincides with the tile graph $T_w$ for every regular sequence $w\in\xo$.
\item Let $w\in\xo$ be a critical sequence, and let $O(w)$ be the set of all critical sequences $v\in\xo$ such
that $g(w)=v$ for some $g\in G$. The Schreier graph $\gr_w$ is constructed by
taking the disjoint union of the orbital tile graphs $T_v$ for $v\in O(w)$ and
connecting two critical sequences $v_1,v_2\in O(w)$ by an edge whenever $s(v_1)=v_2$ for some
$s\in S$.
\end{enumerate}
\end{prop}
\begin{proof}
\textit{1.} If the point $w$ is regular, then the set of vertices of $\gr_w$ is the
cofinality class $Cof(w)$, which is the set of vertices of $T_w$ by
Proposition~\ref{cofinality}. Suppose there is an edge between $v$ and $u$ in the graph
$\gr_w$. Then $s(v)=u$ for some $s\in S$. Since the sequence $w$ is regular, all the
sequences in $Cof(w)$ are regular, and hence there exists a finite beginning $v'$ of
$v$ such that $s|_{v'}=1$. Hence there is an edge between $v$ and $u$ in the tile
graph~$T_w$.

\textit{2.} If the point $w$ is critical, then the set of vertices of $\gr_w$ is the
union of cofinality classes $Cof(v)$ for $v\in O(w)$. Consider an edge $s(v_1)=v_2$ in
$\gr_w$. If this is not an edge of $T_v$ for $v\in O(w)$, then the restriction of $s$
on every beginning of $v_1$ is not trivial. Hence $v_1,v_2$ are critical, and this edge
was added under construction.
\end{proof}
%%%%%%%%%%%%%%%%%%%%%%%%%%%%%%%%%%%%%%%%%%%%%%%%%%%%%%%%%%%%%%%5

The following corollary summarizes the relation between the number of ends of the Schreier graphs $\gr_w$
with the number of ends of the tile graphs $T_w$. It justifies the fact that, for our aims,
it was enough to study the number of ends and connected components in the tile graphs.
%%%%%%%%%%%%%%%%%%%%%%%%%%%%%%%%%%%%%%%%%%%%%%%%%%%%%%%%%%%%%%%
\begin{cor}\label{remarktile}
\begin{enumerate}
  \item If $w$ is a regular sequence, then $\# Ends(\gr_w)=\#Ends(T_{w})$.
   \item If $w$ is critical, then
   $$
   \# Ends(\gr_w)=\sum_{w'\in O(w)} \#Ends(T_{w'}),
$$
where the set $O(w)$ is from Proposition~\ref{prop_Schreier_construction}.
\end{enumerate}
\end{cor}

Using the automata $\A_c$ and $\A_{ic}$ one can construct similar automata for the number of components in the
Schreier graphs $\gr_n$ with a vertex removed. For every state of $\A_c$ or $\A_{ic}$ take the corresponding
partition of the post-critical set and combine components according to the edges $E(\gr\setminus T)$
described in the last paragraph in Section~\ref{subsection_BoundedAutomata}. For example, if
$\{\pcset_i\}_i$ is a state of $\A_{ic}$, then we glue every two components $\pcset_s$ and
$\pcset_t$ if $\{p,q\}\in E(\gr\setminus T)$ for some $p\in \pcset_s$ and $q\in\pcset_t$. We get a new
partition, and we label the state by the number of components in this partition. Basically, we get the
same automata, but vertices may be labeled in a different way.

A case of special interest is when all Schreier graphs have one end.
In our settings of groups generated by bounded automata, our construction enables us to
find a necessary and sufficient condition when all Schreier graphs $\Gamma_w$ have one end.

\begin{theor}\label{all-one-ended}
All orbital Schreier graphs $\Gamma_w$ for $w\in\xo$ have one end if and only if the following two
conditions hold:
\begin{enumerate}
\item all arrows along directed cycles in the automaton $S$ are labeled by $x|x$ for some $x\in X$ (depending on an arrow);
\item all strongly connected components of the automaton $\A_{ic}$ are labeled by $1$ (partitions consisting of one part).
\end{enumerate}
\end{theor}
\begin{proof}
Let us show that the first condition is equivalent to the property that for every $w\in\xo$ the Schreier
graph $\Gamma_w$ and tile graph $T_w$ coincide. If there exists a directed cycle that does not satisfy
condition 1, then there exist two different critical sequences $w,w'$ that are connected in the Schreier
graph, i.e., $s(w)=w'$ for some $s\in S$. In this case $\Gamma_{w}\neq T_{w}$, because by
Proposition~\ref{prop_properties_of_sequences} different critical sequences are non-cofinal, and therefore
belong to different tile graphs. And vice versa, the existence of such critical sequences contradicts
condition 1. Therefore condition 1 implies $O(w)=\{w\}$ for every critical sequence $w$, and thus
$\Gamma_w=T_w$ by Proposition~\ref{prop_Schreier_construction}.

Theorem~\ref{thm_number_of_ends} implies that condition 2 is equivalent to the statement that every
tile graph $T_w$ has one end. Therefore, if the conditions $1$ and $2$ hold, then any Schreier graph
coincides with the corresponding tile graph which has one end.

Conversely, Proposition~\ref{prop_Schreier_construction} implies that if the Schreier graph $\Gamma_{w}$
for a critical sequence $w$ does not coincide with the corresponding tile graph $T_{w}$, then the number
of ends of $\Gamma_w$ is greater than one. (The graph $\Gamma_w$ is a disjoint union of more than one
infinite tile graphs $T_v$, $v\in O(w)$ connected by a finite number of edges.) Therefore, if all Schreier
graphs $\Gamma_w$ have one end, then they should coincide with tile graphs (condition 1 holds) and tile
graphs have one end (condition 2 holds).
\end{proof}

\begin{rem}\label{remarkanoi}
The Hanoi Towers group $H^{(3)}$ \cite{gri_sunik:hanoi} is an example of a group generated by a bounded
automaton for which all orbital Schreier graphs $\Gamma_w$ have one end. On the other hand, this group is
not indicable (since its abelianization is finite) but can be projected onto the infinite dihedral group
\cite{delzant_grigorchuk}. This implies that it contains a normal subgroup $N$ such that the Schreier
coset graph associated with $N$ has two ends. Clearly, for what said above, $N$ does not coincide with the
stabilizer of $w$ for any $w\in X^{\omega}$.
\end{rem}

\subsection{The number of infinite components of tile graphs almost surely}

The structure of the automaton $\A_{ic}$ allows to get results about the measure of
infinite sequences $w\in\xo$ for which the tile graphs $T_w\setminus w$ have a given number of infinite components. We recall that the space $X^{\omega}$ is endowed with the uniform measure.

\begin{rem}\label{rem_subword_and_inf_comp}
It is useful to notice that we can construct a finite word $u\in X^{*}$ such that starting at any state of
the automaton $\A_{ic}$ and following the word $u$ we end in some strongly connected component without
outgoing edges. If these strongly connected components correspond to the partition of the post-critical set
on $k$ parts, then it follows that $\pc(T_n\setminus v_1uv_2)=k$ for all words $v_1,v_2\in X^{*}$ with
$|v_1uv_2|=n$. In other words, $\pc(T_n\setminus v)=k$ for every $v$ that contains $u$ as a subword.
\end{rem}

By Proposition~\ref{prop_infcomp_limit} we get the
description of sequences which correspond to infinite components using the automaton $\A_{ic}$ (but only
for regular and critical sequences).
%~\ref{prop_comp_limit} the graph $T_w\setminus w$ has $k$ components if and only if the path
%in the automaton $\A_{c}$ corresponding to the sequence $w$ will be eventually in some strongly connected
%component labeled by the number $k$. Similarly by Proposition

\begin{cor}\label{cor_inf_comp_almost_all_seq}
The number of infinite connected components of the graph $T_w\setminus w$ is almost surely the same for all
sequences $w\in\xo$. This number coincides with the label of the strongly connected components of the
automaton $\A_{ic}$ without outgoing arrows.
\end{cor}
\begin{proof}
The measure of non-regular sequences is zero. For regular sequences $w\in\xo$ we can use the automaton
$\A_{ic}$ to find the number $\ic(T_w\setminus w)$. Then the corollary follows from
Lemma~\ref{lemma_properties_A_ic} item 2 and the standard fact that the measure of all sequences that are
read along paths in a strongly connected component with an outgoing arrow is zero (for example, this fact
follows from the observation that the adjacency matrix of such a component has spectral radius less than
$|X|$). Another explanation comes from Remark~\ref{rem_subword_and_inf_comp} and the fact that the set of
all sequences $w\in\xo$ that contain a fixed word as a subword is of full measure.
\end{proof}

The corollary does not hold for the number of all connected components of $T_w\setminus w$, see examples in
Section~\ref{Section_Examples}. However, given any number $k$ we can use the automaton $\A_c$ to compute
the measure of the set $C(k)$ %%%%% or $C_k$
 of all sequences $w\in\xo$ such that the graph $T_w\setminus w$ has $k$
components. As shown in the previous proof only strongly connected components without outgoing arrows
contribute the set of sequences with a non-zero measure. Let $\Lambda_k$ be the collection of all strongly
connected components of $\A_c$ without outgoing arrows and labeled by the number $k$. Let $V_k$ be the set
of finite words $v\in X^{*}$ with the property that starting at the initial state and following arrows
labeled by $v$ we end at a component from $\Lambda_k$, and any prefix of $v$ does not satisfy this
property. Then the measure of $C(k)$  %%%%%%%%%%%% or$C_k$
is equal to the sum $\sum_{v\in V_k} |X|^{-|v|}$. Since the automaton $\A_c$ is finite,
this measure is always a rational number and can be computed algorithmically.

\subsection{The number of ends almost surely}\label{Section_number_ends}\label{section_ends_a_s}
%\vspace{0.2cm}\textbf{The number of ends almost surely.}
Corollary~\ref{cor_inf_comp_almost_all_seq} together with Proposition~\ref{prop_ends_limit} imply that the
tile graphs $T_w$ (and thus the Schreier graphs $\gr_w$) have almost surely the same number of ends, and
that this number is equal to the label of the strongly connected components of $\A_{ic}$ without outgoing
arrows. As was mentioned in introduction, this fact actually holds for any finitely generated self-similar
group, which acts transitively on the levels $X^n$ for all $n\in\mathbb{N}$ (see Proposition~6.10 in
\cite{AL}). In our setting of bounded automata we get a stronger description of the sequences $w$ for
which the tile graph $T_w$ has non-typical number of ends.

\begin{prop}\label{prop_more_ends}
There are only finitely many Schreier graphs $\gr_w$ and tile graphs $T_w$ with more than two ends.
\end{prop}
\begin{proof}
Let us prove that the graph $T_w\setminus w$ can have more than two infinite components only for finitely
many sequences $w\in\xo$. Suppose not and choose sequences $w^{(1)},\ldots,w^{(m)}$ such that
$\ic(T_{w^{(i)}}\setminus w^{(i)})\geq 3$, where we take $m$ larger than the number of partitions of the
post-critical set $\pcset$. Choose level $n$ large enough so that all words $w^{(1)}_n,\ldots,w^{(m)}_n$
are different and $\pc(T_n\setminus w_n^{(i)})\geq 3$ for all $i$ (it is possible by
Proposition~\ref{prop_infcomp_limit}). Notice that since the graph $T_n$ is connected, the deletion of
different vertices $w^{(i)}$ produces different partitions of $\pcset$. Indeed, if $\pcset=\sqcup_{i=1}^k
\pcset_i$ with $k\geq 3$ is the partition we got after removing some vertex $v$, then some $k-1$ sets
$\pcset_i$ will be in the same component of the graph $T_n\setminus u$ for any other vertex $u$ (these
$k-1$ sets will be connected through the vertex $v$). We get a contradiction with the choice of number~$m$.

It follows that there are only finitely many tile graphs with more than two ends. This also holds for
Schreier graphs by Proposition~\ref{prop_Schreier_construction}.
\end{proof}

\begin{cor}\label{cor_>2ends_pre_periodic}
The Schreier graphs $\gr_w$ and tile graphs $T_w$ can have more than two ends only for pre-periodic
sequences $w$.
\end{cor}
\begin{proof}
Since the graph $T_w\setminus w$ can have more than two infinite components only for finitely many
sequences $w$, we get that, in the limit in Proposition~\ref{prop_ends_limit}, the sequence $\sigma^n(w)$
attains a finite number of values. Hence $w$ is pre-periodic.
\end{proof}

Example with $IMG(z^2+i)$ in Section~\ref{Section_Examples} shows that the Schreier graph $\gr_w$ and the
tile graph $T_w$ may have more than two ends even for regular sequences $w$.

\begin{cor}\label{cor_one_or_two_ends_a_s}
The tile graphs $T_w$ and Schreier graphs $\gr_w$ have the same number of ends for almost all sequences $w\in\xo$, and this number is equal to one or two.
\end{cor}

\subsection{Two ends almost surely}\label{section_two_ends_a_s}

In this section we describe bounded automata for which Schreier graphs $\Gamma_w$ and tile graphs $T_w$
have almost surely two ends. Notice that in this case the post-critical set $\pcset$ cannot consist of one
element (actually, every finitely generated self-similar group with $|\pcset|=1$ is finite and cannot act
transitively on $X^n$ for all $n$).

\begin{lemma}\label{lemma_as_two_ends_|P|=2}
If the Schreier graphs $\Gamma_w$ (equivalently, the tile graphs $T_w$) have two ends for almost all
$w\in\xo$, then $|\pcset|=2$.
\end{lemma}
\begin{proof}
We pass to a power of the alphabet so that every post-critical sequence is of the form $y^{-\omega}$ or
$y^{-\omega}x$ for some letters $x,y\in X$ and different post-critical sequences end with different letters.
In particular, every subset $\pcset\times\{x\}$ for $x\in X$ of the model graph $M$ contains at most one
post-critical vertex of $M$.

We again pass to a power of the alphabet so that for every nontrivial element $s\in S$ there exists a
letter $x\in X$ such that $s(x)\neq x$ and $s|_x=1$. Then every post-critical sequence $p\in\pcset$ appears
in some edge $\{(p,*), (*,*)\}$ of the model graph. Indeed, if the pair $p|q$ is read along a left-infinite
path in the automaton $S\setminus\{1\}$ that ends in a nontrivial state $s$, then the pair
$\{(p,x),(q,s(x))\}$ belongs to the edge set $E$ of the graph $M$.

Now suppose that tile graphs have almost surely two ends. Then the strongly connected components without outgoing
arrows in the automaton $\A_{ic}$ correspond to the partitions of the post-critical set $\pcset$ on two
parts (see Corollary \ref{cor_inf_comp_almost_all_seq}). In particular, there is no state corresponding to the partition of $\pcset$ with one part, because
such a partition would form a strongly connected component without outgoing arrows (see the construction of
$\A_{ic}$). We will use the fact that all paths in the automaton $\A_{ic}$ starting at any partition
$\pcset=\pcset_1\sqcup \pcset_2$ end in partitions of $\pcset$ on two parts (we cannot get more parts).

Let us construct an auxiliary graph $\overline{M}$ as follows: take the model graph $M$ and for each $x\in
X$ add edges between all vertices in the subset $\pcset\times \{x\}$. We will prove that the graph
$\overline{M}$ is an ``interval'', i.e., there are two vertices of degree one and the other vertices have
degree two, and that two end vertices of $\overline{M}$ are the only post-critical vertices. First, let us
show that there are only two subsets $\pcset \times \{x\}$ for $x\in X$ such that the graph
$\overline{M}\setminus \pcset\times \{x\}$ is connected. Suppose that there are three such subsets
$\pcset\times \{x\}$, $\pcset\times \{y\}$, $\pcset\times \{z\}$.

Fix any partition $\pcset=\pcset_1\sqcup \pcset_2$
that corresponds to some state of the automaton $\A_{ic}$. Consider the arrow in the automaton $\A_{ic}$
starting at $\pcset_1\sqcup \pcset_2$ and labeled by $x$. This arrow ends in the partition
$\pcset=\pcset_1^{(x)}\sqcup \pcset_2^{(x)}$ with two parts. Recall how we construct the partition
$\pcset_1^{(x)}\sqcup \pcset_2^{(x)}$ using the graph $M_{\pcset_1\sqcup \pcset_2,x}$, and notice that
$M_{\pcset_1\sqcup \pcset_2,x}\setminus\pcset\times \{x\}$ coincides with $\overline{M}\setminus \pcset\times \{x\}$. Using the
assumption that the graph $\overline{M}\setminus \pcset\times \{x\}$ is connected, we get that one of the sets
$\pcset_i^{(x)}$ is a subset of $\pcset\times \{x\}$. Since $\pcset\times \{x\}$ contains at most one post-critical
vertex, the part $\pcset_i^{(x)}$ consists of precisely one element (post-critical vertex), which we denote
by $a\in\pcset$, i.e., here $\pcset_i^{(x)}=\{a\}$. By the same reason the subsets $\pcset\times \{y\}$ and
$\pcset\times \{z\}$ also contain some post-critical vertices $b$ and $c$. Notice that the last letters of the
sequences $a, b, c$ are $x,y,z$ respectively. We can suppose that the sequences $az$ and $bz$ are different
from the sequence $c$ (over three post-critical sequences there are always two with this property).
Consider the arrow in the automaton $\A_{ic}$ starting at the partition $\pcset_1^{(x)}\sqcup
\pcset_2^{(x)}=\{a\}\sqcup \pcset\setminus \{a\}$ and labeled by $z$. This arrow should end in the
partition of $\pcset$ on two parts. Since $az$ and $c$ are different, the post-critical vertex $c$ of
$M_{\pcset_1\sqcup \pcset_2,x}$ belongs to the subset $(\pcset\setminus \{a\})\times \{z\}$. Further, since $c$ is the unique
post-critical vertex in $\pcset\times \{z\}$, there should be no edges connecting the subset $(\pcset\setminus
\{a\})\times \{z\}$ with its outside in the graph $M_{\pcset_1\sqcup \pcset_2,x}$ (otherwise all post-critical vertices will be
in the same component). Hence the only edges of the graph $\overline{M}$ going outside the subset
$\pcset\times \{z\}$ should be at the vertex $(a,z)$. Applying the same arguments to the partition $\{b\}\sqcup
\pcset\setminus\{b\}$, we get that this unique vertex should be $(b,z)$. Hence $a=b$ and we get a
contradiction.

So let $\pcset\times \{x\}$ and $\pcset\times \{y\}$ be the two subsets such that their complements in the graph
$\overline{M}$ are connected. Let $a$ and $b$ be the post-critical vertices in $\pcset\times \{x\}$ and
$\pcset\times \{y\}$ respectively. By the same arguments as above, the subset $\pcset\times \{x\}$ has a unique
vertex which is adjacent to a vertex from $\overline{M}\setminus \pcset\times \{x\}$, and this vertex is of the
form $(a,x)$ or $(b,x)$. The same holds for the subset $\pcset\times \{y\}$. Every other component
$\pcset\times \{z\}$ contains precisely two vertices $(a,z)$ and $(b,z)$, which have edges going outside the
component $\pcset\times \{z\}$. However every post-critical sequence appears in one of such edges (see our
assumption in the second paragraph of the proof). Hence the post-critical set contains precisely two
elements and the structure of the graph $\overline{M}$ follows.
\end{proof}

\begin{cor}\label{cor_one_end_pcset_3_points}
If the post-critical set $\pcset$ contains at least three sequences, then the Schreier graphs $\Gamma_w$
and tile graphs $T_w$ have almost surely one end.
\end{cor}

The following example shows that almost all Schreier graphs may have two ends for a contracting group
generated by a non-bounded automaton, i.e., by an automaton with infinite post-critical set.

\begin{example}
Consider the self-similar group $G$ over $X=\{0,1,2\}$ generated by the transformation $a$, which is given by the recursion  $a=(a^2,1,a^{-1})(0,1,2)$
(see Example~7.6 in \cite{bhn:aut_til}). The group $G$ is self-replicating and contracting with nucleus
$\nucl=\{1,a^{\pm 1}, a^{\pm 2}\}$, but the generating automaton is not bounded and the post-critical set
is infinite. Every Schreier graph $\gr_w$ with respect to the generating set $\{a,a^{-1}\}$ is a line and
has two ends.
\end{example}

\begin{theor}\label{th_classification_two_ends}
Almost all Schreier graphs $\gr_w$ (equivalently, tile graphs $T_w$) have two ends if and only if the
automaton $S$ brought to the basic form (see Section~\ref{subsection_BoundedAutomata}) is one of the
following.
\begin{enumerate}
  \item The automaton $S$ consists of the adding machine, its inverse, and the trivial element, where the adding machine is an
  element of type  I  with a transitive action
on $X$ (see Figure~\ref{fig_Bounded_PCS2}, where all edges not shown
in the figure go to the identity state, and the letters $x$ and $y$ are different).

  \item There exists an order on the alphabet $X=\{x=x_1, x_2, \ldots, x_m=y\}$ such that one of
  the following cases holds.
  \begin{enumerate}
    \item The automaton $S$ consists of elements of types II and II$'$ (see Figure~\ref{fig_Bounded_PCS2});
    every pair $\{x_{2i},x_{2i+1}\}$ is an orbit of the action of some element of type II and
    all nontrivial orbits of such elements on $X$ are of this form; also every
pair $\{x_{2i-1},x_{2i}\}$ is an orbit of the action of some element of type II$'$ and all nontrivial
orbits of such elements on $X$ are of this form (in particular, $|X|$ is an odd number).

\begin{figure}
\begin{center}
\psfrag{xy}{$x|y$} \psfrag{yy}{$y|y$} \psfrag{xx}{$x|x$}  \psfrag{yx}{$y|x$} \psfrag{TI}{Type I}
\psfrag{TI2}{Type I$'$} \psfrag{TII}{Type II} \psfrag{TII2}{Type II$'$} \psfrag{TIII}{Type III}
\psfrag{TIII2}{Type III$'$} \epsfig{file=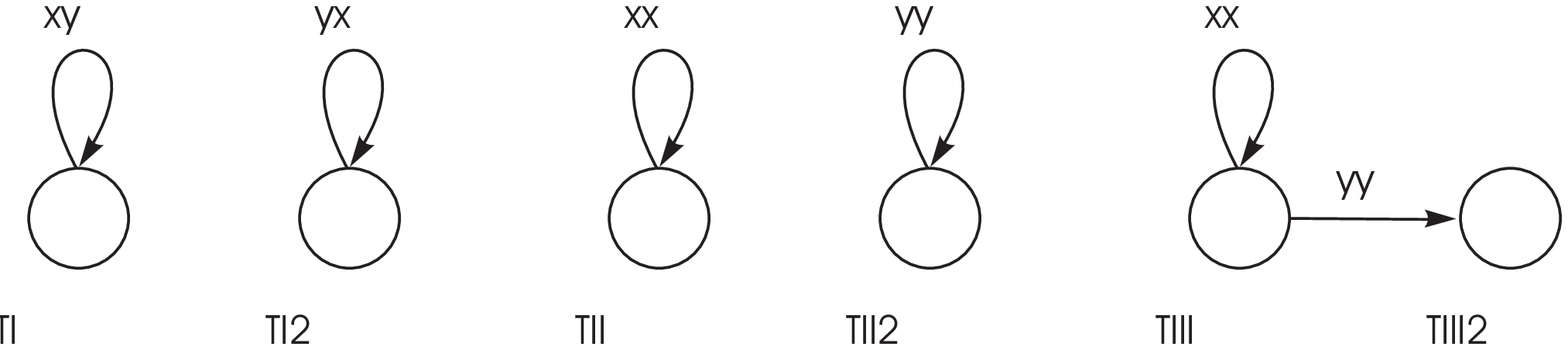,height=80pt}
\caption{Bounded automata with
$|\pcset|=2$}
\label{fig_Bounded_PCS2}
\end{center}
\end{figure}

    \item The automaton $S$ consists of elements of types II, III, and III$'$; every pair $\{x_{2i},x_{2i+1}\}$ is an
    orbit of the action of some element of type II or $III$ and all nontrivial orbits of such elements on $X$
    are of this form; also every pair $\{x_{2i-1},x_{2i}\}$ is an orbit of the action of some element of type III$'$ and
all nontrivial orbits of such elements on $X$ are of this form (in particular, $|X|$ is an even number).
  \end{enumerate}
\end{enumerate}
Moreover, in this case, all Schreier graphs $\gr_w$ are lines except for two Schreier graphs $\gr_{x^{\omega}}$ and $\gr_{y^{\omega}}$ in Case 2 (a), and one Schreier graph $\gr_{x^{\omega}}$ in Case 2 (b), which are rays.
\end{theor}
\begin{proof}
Recall the definition of the basic form of a bounded automaton from
Section~\ref{subsection_BoundedAutomata}. If a bounded automaton is in the basic form, its post-critical
set has size two, and every state has an incoming arrow, then it is not hard to see that the automaton can
contain only the states of six types shown in Figure~\ref{fig_Bounded_PCS2}.

We will be using the fact proved in Lemma~\ref{lemma_as_two_ends_|P|=2} that the modified model graph
$\overline{M}$ is an interval. There are two cases that we need to treat a little bit differently depending
on whether both post-critical sequences are periodic or not.

Consider the case when both post-critical sequences are periodic, here $\pcset=\{x^{-\omega},
y^{-\omega}\}$. In this case the automaton $S$ can contain only the states of types I, I$'$, II, and II$'$.
Suppose there is a state $a$ of type I. It contributes the edges $\{ (x^{-\omega},z),(y^{-\omega},a(z)) \}$
to the graph $\overline{M}$ for every $z\in X$. If there exists a nontrivial orbit of the action of $a$ on
$X$, which does not contain $x$, then it contributes a cycle to the graph $\overline{M}$. If there exists a
fixed point $a(z)=z$, then under construction of the automaton $\A_{ic}$ starting at the partition
$\pcset=\{x^{-\omega}\}\sqcup\{y^{-\omega}\}$ and following the arrow labeled by $z$ we get a partition
with one part. Hence the element $a$ should act transitively on $X$ (it is the adding machine). Every other
element of type I should have the same action on $X$, and hence coincide with $a$, otherwise we would got a
vertex in the graph $\overline{M}$ of degree $\geq 3$. Every element $b$ of type I$'$ contributes the edges
$\{(x^{-\omega},b(z)),(y^{-\omega},z) \}$ to the graph $\overline{M}$. It follows that the action of $b$ on
$X$ is the inverse of the action of $a$ (otherwise we would got a vertex of $\overline{M}$ of degree $\geq
3$), and hence $b$ is the inverse of $a$. If the automaton $S$ additionally contains a state of type II or
II$'$, then there is an edge $\{(x^{-\omega},z_1),(x^{-\omega},z_2)\}$ or
$\{(y^{-\omega},z_1),(y^{-\omega},z_2)\}$ in the graph $\overline{M}$ for some different letters
$z_1,z_2\in X$. We get a vertex of degree $\geq 3$, contradiction. Hence, in this case, the automaton $S$
consists of the adding machine, its inverse, and the identity state.

Suppose $S$ does not contain states of types I and I$'$. Since the post-critical set is equal to
$\pcset=\{x^{-\omega}, y^{-\omega}\}$ the automaton $S$ contains states $a$ and $b$ of types II and II$'$
respectively. These elements contribute edges $\{(x^{-\omega},z),(x^{-\omega},a(z))\}$ and
$\{(y^{-\omega},z),(y^{-\omega},b(z))\}$ to the graph $\overline{M}$. Since the graph $\overline{M}$ should
be an interval, these edges should consequently connect all components $\pcset\times z$ for $z\in X$ (see
Figure~\ref{fig_ModelLine_Case 2a}). It follows that there exists an order on the alphabet such that item
$(a)$ holds.

\begin{figure}
\begin{center}
\psfrag{1}{$x^{-\omega}$} \psfrag{2}{$y^{-\omega}x$} \psfrag{3}{$y^{-\omega}x_2$}
\psfrag{4}{$x^{-\omega}x_2$} \psfrag{5}{$x^{-\omega}x_3$} \psfrag{6}{$y^{-\omega}x_3$}
\psfrag{7}{$x^{-\omega}y$} \psfrag{8}{$y^{-\omega}$}
 \epsfig{file=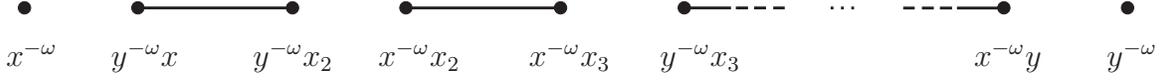,height=26pt}
\caption{Case 2 (a)}
\label{fig_ModelLine_Case 2a}
\end{center}
\end{figure}

\begin{figure}
\begin{center}
\psfrag{1}{$x^{-\omega}$} \psfrag{2}{$x^{-\omega}yx$} \psfrag{3}{$x^{-\omega}yx_2$}
\psfrag{4}{$x^{-\omega}x_2$} \psfrag{5}{$x^{-\omega}x_3$} \psfrag{6}{$x^{-\omega}yx_3$}
\psfrag{7}{$x^{-\omega}yy$} \psfrag{8}{$x^{-\omega}y$}
\epsfig{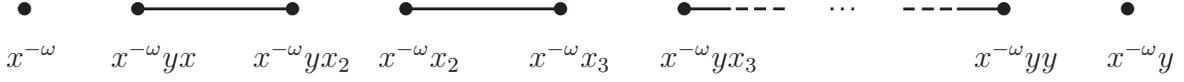}
\caption{Case 2 (b)}
\label{fig_ModelLine_Case 2b}
\end{center}
\end{figure}

Consider the case $\pcset=\{x^{-\omega}, x^{-\omega}y\}$. In this case the automaton $S$ can consist only
of states of types II, III and III$'$. Each state of type II or III contributes edges
$\{(x^{-\omega},z),(x^{-\omega},a(z))\}$ to the graph $\overline{M}$. Each state of type III$'$ contributes
edges $\{(x^{-\omega}y,z),(x^{-\omega}y,b(z))\}$. These edges should consequently connect all components
$\pcset\times z$ for $z\in X$ (see Figure~\ref{fig_ModelLine_Case 2b}). It follows that there exists an
order on the alphabet such that item $(b)$ holds.

For the converse, one can directly check the following facts. In item 1, every Schreier graph $\gr_w(G,S)$
is a line. In Case 2 (a), the Schreier graphs $\gr_{x^{\omega}}$ and $\gr_{y^{\omega}}$ are rays, while all
the other Schreier graphs are lines. In Case 2 (b), the Schreier graph $\gr_{x^{\omega}}$ is a ray, while
all the other Schreier graphs are lines.
\end{proof}

\begin{example}
Dihedral group
\end{example}

\begin{example}
The Grigorchuk group is a nontrivial example satisfying the conditions of the theorem. It is generated by
the automaton $S$ shown in Figure~\ref{fig_Grigorchuk_Aut}. After passing to the alphabet
$\{0,1\}^3\leftrightarrow X=\{0,1,2,3,4,5,6,7\}$, the automaton $S$ consists of the trivial state and the elements $a,b,c,d$, which are given by the following recursions:
\begin{eqnarray*}
a&=&(1,1,1,1,1,1,1,1)(0,4)(1,5)(2,6)(3,7)\\
b&=&(1,1,1,1,1,1,1,b)(0,2)(1,3)(4,5)\\
c&=&(1,1,1,1,1,1,a,c)(0,2)(1,3)\\
d&=&(1,1,1,1,1,1,a,d)(4,5)
\end{eqnarray*}
We see that this automaton satisfies Case 2 (b) of the theorem when we choose the order $6,2,0,4,5,1,3,7$
on $X$. The Schreier graph $\gr_{7^{\omega}}$ is a ray, while the other orbital Schreier graphs are lines.
\end{example}

\begin{figure}
\begin{center}
\psfrag{id}{$id$} \psfrag{2}{$a$} \psfrag{3}{$b$} \psfrag{4}{$d$} \psfrag{5}{$c$} \psfrag{6}{$0|1$}
\psfrag{7}{$1|0$} \psfrag{8}{$0|0$} \psfrag{9}{$0|0$} \psfrag{10}{$1|1$} \psfrag{11}{$1|1$}
\psfrag{12}{$1|1$} \psfrag{13}{$0|0$} \psfrag{14}{$0|0$} \psfrag{15}{$1|1$}
\epsfig{file=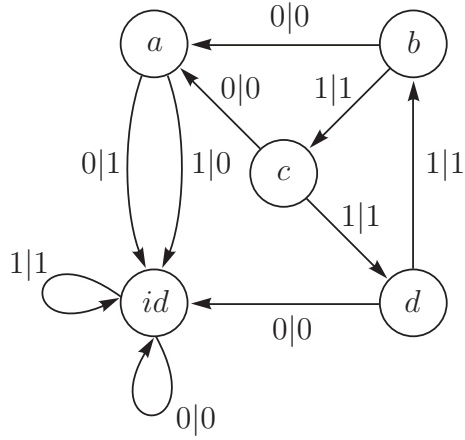,height=160pt}
\caption{The generating automaton of the Grigorchuk group}
\label{fig_Grigorchuk_Aut}
\end{center}
\end{figure}

In what follows below, we give an algebraic characterization of the automaton groups acting on the binary tree, whose orbital Schreier graphs have two ends. It turns out that such groups are those whose nuclei are given by the automata defined by  \v{S}uni\'{c} in \cite{Zoran}. In order to show this correspondence we sketch the construction of the groups $G_{\omega, \rho}$ as in \cite{Zoran}.

Let $A$ and $B$ be the abelian groups $\mathbb{Z}/2\mathbb{Z}$ and $(\mathbb{Z}/2\mathbb{Z})^k$ respectively. We think of $A$ as the field of $2$ elements and of $B$ as the $k$-dimensional vector space over this field. Let $\rho:B\to B$ be an automorphism of $B$ and $\omega:B\to A$ a surjective homomorphism. We define the action of elements of $A$ and $B$ on the binary tree $\{0,1\}^{*}$ as follows: the nontrivial element $a\in A$ only changes the first letter of input words, i.e., $a=(1,1)(0,1)$; the action of $b\in B$ is given by the recursive rule $b=(\omega(b),\rho(b))$. The automorphism group generated by the action of $A$ and $B$ is denoted by $G_{\omega, \rho}$. Notice that the group $G_{\omega, \rho}$ is generated by a bounded automaton over the binary alphabet $X=\{0,1\}$, which we denote by $A_{\omega, \rho}$. If the action of $B$ is faithful, the group $G_{\omega, \rho}$ can be given by an invertible polynomial over the field with two elements, which corresponds to the action of $\rho$ on $B$. Examples include the infinite dihedral group $D_{\infty}$ given by polynomial $x+1$ and the Gigorchuk group given by polynomial $x^2+x+1$.

The following result puts in relation the groups $G_{\omega, \rho}$ and the groups whose Schreier graphs have almost surely two ends.

%It is not hard to prove that the Grigorchuk group is contained in this class of automata groups. In Proposition 2 of \cite{Zoran} is shown that, if the action of $B$ is faithful on the rooted binary tree, then one can find a minimal polynomial $p(x)$ (corresponding to the cyclic action of  $\rho$ on $B$) characterizing the group $G_{\omega, \rho}$. From now on, following \cite{Zoran} we denote such a group $G_{p(x)}$.  More precisely, given an element $b\in B$ we consider its orbit under $\rho$. There exists a minimal $d$ such that $\{b, \rho(b), \rho^2(b),\ldots, \rho^d(b)\}$ are linearly dependent vectors. Their combination which gives the zero vector gives rise to the minimal polynomial $p(x)$. Notice that, since $\rho$ is an isomorphism and $B$ is finite, there exists a $n$ such that $\rho^n(b)=b$, for any $b\in B$. The groups $G_{p(x)}$ are spherically transitive and branch. Notice that one can naturally build a finite automaton $\A_{p(x)}$ generating $G_{p(x)}$ starting from the maps $\omega$ and $\rho$. Such automaton is bounded and concides with its nucleus.

\begin{theor}\label{th_Sch_2ended_binary}
Let $G$ be a group generated by a bounded automaton over the binary alphabet $X=\{0,1\}$. Almost all Schreier graphs $\Gamma_w(G)$ have two ends if and only if the nucleus of $G$ either consists of the adding machine, its inverse and the identity element, or is equal to one of the automata $\A_{\omega, \rho}$ up to switching the letters $0\leftrightarrow 1$ of the alphabet.
\end{theor}
\begin{proof}
Suppose that the Schreier graphs $\Gamma_w(G)$ have two ends for almost all sequences $w\in\xo$ and let
$\nucl$ be the nucleus of the group $G$. By Lemma~\ref{lemma_as_two_ends_|P|=2} the post-critical set
$\pcset$ of the group contains exactly two elements. If both post-critical elements are periodic, then the
nucleus $\nucl$ consists of the adding machine, its inverse and the identity element (see the proof of
Theorem~\ref{th_classification_two_ends}).

Let us consider the case $\pcset=\{x^{-\omega}, x^{-\omega}y\}$.
We can assume that $x=1$ and $y=0$ so that $\pcset=\{1^{-\omega},1^{-\omega}0\}$.
It follows that every arrow along a cycle in $\nucl$ is labeled by $1|1$ except for a loop at the trivial state labeled by $0|0$. The nucleus $\nucl$ contains only one nontrivial finitary element, namely $a=(1,1)(0,1)$, since otherwise there would be a post-critical sequence with preperiod of length two.

Put $A=\{1,a\}$ and $B=\nucl\setminus\{a\}$. The set $B$ consists exactly of those elements from $G$ that belong to cycles. For every $b\in B$ we have $b(1)=1$ and $b|_1\in B$, $b|_0\in A$. It follows that all nontrivial elements of $B$ have order two. Let us show that $B$ is a subgroup of $G$. For any  $b_1,b_2\in B$ there exists $k$ such that $b_i|_{1^k}=b_i$ for $i=1,2$. Hence $b_1b_2|_{1^{k}}=b_1b_2$. Therefore $b_1b_2$ belongs to the nucleus $\nucl$ and thus $b_1b_2\in B$. It follows that $B$ is a group, which is isomorphic to $(\mathbb{Z}/2\mathbb{Z})^m$ for certain $m$. The map $\rho: b\mapsto b|_1$ is a homomorphism from $B$ to $B$ and is bijective, because elements of $B$ form cycles in the nucleus. The map $\omega: b\rightarrow b|_0$ is a surjective homomorphism from $B$ to $A$. We have proved that the nucleus $\nucl$ is exactly the automaton $A_{\omega, \rho}$.

%Suppose that the Schreier graphs $\Gamma_w(G)$ have almost surely two ends for all sequences
%$w\in\xo$. Then the group $G$ is infinite (otherwise all Schreier graphs would be finite and have no ends). Since the alphabet is binary, the Schreier graphs $\Gamma_n(G)$ are connected by Proposition~?? in \cite{} (as well as the tile graphs $T_n(G)$). Let $\nucl$ be the nucleus of the group $G$ and $H$ be the subgroup of $G$ generated by $\nucl$. Since $G$ is infinite, $H$ is infinite too, and therefore the tile graphs $T_n(H,\nucl)$ are connected. Hence, the nucleus $\nucl$ satisfies our technical assumptions.   By Lemma~\ref{lemma_as_two_ends_|P|=2} we have the post-critical set

%Since $H$ is a subgroup of $G$, the Schreier graphs $\Gamma_w(H,\nucl)$ also have almost surely two ends for all sequences
%$w\in\xo$.

On the other hand, let us consider one of the groups $G_{\omega, \rho}$. Each element $b\in B$ has a
cyclic $\rho$-orbit. In the language of automata this means that $b$ belongs to a cycle in the automaton
$A_{\omega, \rho}$. We have $b|_1\in B$ and $b|_0\in A$ for every $b\in B$. It follows that if $b(u)=v$
and $u\neq v$ then $u$ and $v$ are of the form $u=1^l00w$, $v=1^l01w$ or $u=1^l01w$, $v=1^l00w$,
$l\in\mathbb{N}\cup\{0\}$. Therefore the Schreier graphs $\Gamma_n(G)$ are intervals and the orbital
Schreier graphs $\Gamma_w(G)$ have two ends for almost all sequences $w\in\xo$.
\end{proof}

\section{Cut-points of tiles and limit spaces}\label{Section_Cut-points}

In this section we first recall the construction of the limit space and tiles of a self-similar group (see
\cite{self_sim_groups, ssgroups_geom} for more details). Then we show how to describe the cut-points of limit spaces and tiles of self-similar groups generated by bounded automata.

%As before, this can be done by using a finite recognition automaton.

\subsection{Limit spaces and tiles of self-similar groups}

Let $G$ be a contracting self-similar group with nucleus $\nucl$.

\begin{defi}\label{limitspacedefi}
The \textit{limit space} $\lims$ of the group $G$ is the quotient of the space $\xmo$ by the equivalence
relation, where two sequences $\ldots x_2x_1$ and $\ldots y_2y_1$ are equivalent if there exists a
left-infinite path in the nucleus $\nucl$ labeled by the pair $\ldots x_2x_1|\ldots y_2y_1$.
\end{defi}

The limit space $\lims$ is compact, metrizable, finite-dimensional space. If the group $G$ is finitely
generated and self-replicating, then the space $\lims$ is path-connected and locally path-connected (see
\cite[Corollary~3.5.3]{self_sim_groups}). The shift map on the space $X^{-\omega}$ induces a continuous
surjective map $\si:\lims\rightarrow\lims$. The limit space $\lims$ comes together with a natural Borel
measure $\mu$ defined as the push-forward of the uniform Bernoulli measure on $\xmo$. The dynamical system
$(\lims,\si,\mu)$ is conjugate to the one-sided Bernoulli $|X|$-shift (see \cite{bk:meas_limsp}).

\begin{defi}
The \textit{limit  $G$-space} $\limGs$ of the group $G$ is the quotient of the space $\xmo\times G$
equipped with the product topology of discrete sets by the equivalence relation, where two sequences
$\ldots x_2x_1\cdot g$ and $\ldots y_2y_1\cdot h$ of $\xmo\times G$ are equivalent if there exists a
left-infinite path in the nucleus $\nucl$ that ends in the state $hg^{-1}$ and is labeled by the pair
$\ldots x_2x_1|\ldots y_2y_1$.
\end{defi}

The space $\limGs$ is metrizable and locally compact. The group $G$ acts properly and cocompactly  on the
space $\limGs$ by multiplication from the right. The quotient of $\limGs$ by the action of $G$ is the
space $\lims$.

\begin{defi}
The image of $\xmo\times \{1\}$ in the space $\limGs$ is called the \textit{tile} $\tile$ of the group
$G$. The image of $\xmo v\times \{1\}$ for $v\in X^n$ is called the \textit{tile} $\tile_v$ of $n$-th
level.
\end{defi}

Alternatively, the tile $\tile$ can be described as the quotient of $\xmo$ by the equivalence relation,
where two sequences $\ldots x_2x_1$ and $\ldots y_2y_1$ are equivalent if and only if there exists a path
in the nucleus $\nucl$ that ends in the trivial state and is labeled by the pair $\ldots x_2x_1|\ldots
y_2y_1$. The push-forward of the uniform measure on $\xmo$ defines a measure on $\tile$. The tile $\tile$ covers the limit $G$-space $\limGs$ under the action of $G$.

The tile $\tile$ decomposes in the union $\cup_{v\in X^n} \tile_v$ of the tiles of $n$-th level for every
$n$. All tiles $\tile_v$ are compact and homeomorphic to $\tile$. Two tiles $\tile_v$ and $\tile_u$ of the
same level $v,u\in X^{n}$ have nonempty intersection if and only if there exists $h\in\nucl$ such that
$h(v)=u$ and $h|_v=1$ (see \cite[Proposition~3.3.5]{self_sim_groups}). This is precisely how we connect
vertices in the tile graph $T_n(G, \nucl)$ with respect to the nucleus. Hence the graphs $T_n(G,\nucl)$
can be used to approximate the tile $\tile$, which justifies the term ``tile'' graph. The tile $\tile$ is
connected if and only if all the tile graphs $T_n=T_n(G,\nucl)$ are connected (see
\cite[Proposition~3.3.10]{self_sim_groups}); in this case also $\tile$ is path-connected and locally
path-connected.

%\begin{defi}(Nekrashevych)
%The \textit{tile} $\tile$ of the group is the quotient of $\xmo$ by the equivalence relation, where two
%sequences $\ldots x_2x_1$ and $\ldots y_2y_1$ are equivalent if and only if there exists a path in the
%nucleus $\nucl$ that ends in the trivial state and is labeled by the pair $\ldots x_2x_1|\ldots y_2y_1$.
%The quotient of $\xmo v$ for $v\in X^n$ is called the \textit{tile} $\tile_v$ of $n$-th level.
%\end{defi}

%The tile $\tile$ decomposes in the union $\cup_{v\in X^n} \tile_v$ of the tiles of $n$-th level for every
%$n$. All tiles $\tile_v$ are compact and homeomorphic to $\tile$.

\begin{defi}
A contracting self-similar group $G$ satisfies the \emph{open set condition} if for any element $g$ of the
nucleus $\nucl$ there exists a word $v\in X^*$ such that $g|_v=1$, i.e., in the nucleus $\nucl$ there is a
path from any state to the trivial state.
\end{defi}

If a group satisfies the open set condition, then the tile $\tile$ is the closure of its interior, and any
two different tiles of the same level have disjoint interiors; otherwise for large enough $n$ there exists
a tile $\tile_v$ for $v\in X^n$ which is covered by other tiles of $n$-th level (see
\cite[Proposition~3.3.7]{self_sim_groups}).

Recall that the post-critical set $\pcset$ of the group is defined as the set of all sequences that can be
read along left-infinite paths in $\nucl\setminus\{1\}$. Therefore, under the open set condition, the
boundary $\partial\tile$ of the tile $\tile$ consists precisely of points represented by the post-critical
sequences. Under the open set
condition, the limit space $\lims$ can be obtained from the tile $\tile$ by gluing some of its boundary
points. Namely, we need to glue two points represented by (post-critical) sequences $\ldots x_2x_1$ and
$\ldots y_2y_1$ for every path in $\nucl\setminus\{1\}$ labeled by $\ldots x_2x_1|\ldots y_2y_1$.

Every self-similar group generated by a bounded automaton is contracting as shown in \cite{bondnek:pcf},
and we can consider the associated limit spaces and tiles. Note that every bounded automaton satisfies the open set condition. The limit spaces of groups generated by bounded
automata are related to important classes of fractals: post-critically finite and finitely-ramified
self-similar sets (see \cite[Chapter~IV]{PhDBondarenko}). Namely, for a contracting self-similar group $G$
with nucleus $\nucl$ the following statements are equivalent: every two tiles of the same level have
finite intersection (the limit space $\lims$ is finitely-ramified); the post-critical set $\pcset$ is
finite (the limit space $\lims$ is post-critically finite); the nucleus $\nucl$ is a bounded automaton (or
the generating automaton of the group is bounded). Under the open set condition, the above statements are
also equivalent to the finiteness of the tile boundary $\partial\tile$.\vspace{0.2cm}

\textbf{Iterated monodromy groups.} Let $f\in \mathbb{C}(z)$ be a complex rational function of degree
$d\geq 2$ with finite post-critical set $P_f$. Then $f$ defines a $d$-fold partial self-covering $f:
f^{-1}(\M)\rightarrow\M$ of the space $\M=\Rs\setminus P_f$. Take a base point $t\in\M$ and let $T_t$ be
the tree of preimages $f^{-n}(t)$, $n\geq 0$, where every vertex $z\in f^{-n}(t)$ is connected by an edge
to $f(z)\in f^{-n+1}(t)$. The fundamental group $\pi_1(\M,t)$ acts by automorphisms on $T_t$ through the
monodromy action on every level $f^{-n}(t)$. The quotient of $\pi_1(\M,t)$ by the kernel of its action on
$T_t$ is called the \textit{iterated monodromy group} $IMG(f)$ of the map $f$. The group $IMG(f)$ is
contracting self-similar group and the limit space $\lims$ of the group $IMG(f)$ is homeomorphic to the
Julia set $J(f)$ of the function $f$ (see \cite[Section~6.4]{self_sim_groups} for more details). Moreover,
the limit dynamical system $(\lims[IMG(f)],\si,\muls)$ is conjugated to the dynamical system
$(J(f),f,\mu_f)$, where $\mu_f$ is the unique $f$-invariant probability measure of maximal entropy on the
Julia set $J(f)$ (see \cite{bk:meas_limsp}).

\subsection{Cut-points of tiles and limit spaces}\label{cut-points section}
In this section we show how the number of connected components in the orbital Schreier and tile graphs
with a vertex removed is related to the number of connected components in the limit space and tile with a point removed. This allows us to get a description of cut-points using a finite acceptor automata as in Proposition~\ref{prop_infcomp_A_ic}.

Let $G$ be a self-similar group generated by a bounded automaton. We assume that the tile $\tile$ is
connected. Then the nucleus $\nucl$ of the group is a bounded automaton, every state of $\nucl$ has an
incoming arrow, and all the tile graphs $T_n=T_n(G,\nucl)$ are connected. Hence we are in the settings of
Section~\ref{Section_Ends}, and we can apply its results to the tile graphs $T_n$. Since the limit space $\lims$ is obtained from the tile $\tile$ by gluing finitely many specific
boundary points (the post-critical set $\pcset$ is finite), it is sufficient to consider the problem for
the tile $\tile$, in analogy to what we made before for the Schreier and tile graphs.

\subsubsection{Boundary, critical and regular points}

The tile $\tile$ decomposes into the union $\cup_{x\in X}\tile_x$, where each tile $\tile_x$ is
homeomorphic to $\tile$ under the shift map. It follows that, if we take a copy $(\tile,x)$ of the tile
$\tile$ for each $x\in X$ and glue every two points $(t_1,x)$ and $(t_2,y)$ with the property that there
exists a path in the nucleus $\nucl$ that ends in the trivial state and is labeled by $\ldots x_2x_1x|\ldots
y_2y_1y$, and the sequences $\ldots x_2x_1$ and $\ldots y_2y_1$ represent the points $t_1$ and $t_2$
respectively, then we get a space homeomorphic to the tile $\tile$. This is an analog of the construction
of tile graphs $T_n$ given in Theorem~\ref{th_tile_graph_construction}. The edges of the model graph $M$
now indicate which points of the copies $(\tile,x)$ should be glued.

We consider the tile $\tile$ as its own topological space (with the induced topology from the space
$\limGs$), and every tile $\tile_v$ for $v\in X^{*}$ as a subset of $\tile$ with induced topology. Hence
the boundary of $\tile$ is empty, but the points represented by post-critical sequences we still call the
\textit{boundary points} of the tile. Every point in the intersection of different tiles
$\tile_v\cap\tile_u$ of the same level $|v|=|u|$ we call \textit{critical}. These points are precisely the
boundary points of the tiles $\tile_v$ for $v\in X^{*}$, and they are represented by sequences of the form
$pv$ for $p\in\pcset$ and $v\in X^{*}$. In particular, the number of critical points is countable, and
hence they are of measure zero. All other points of $\tile$ we call \textit{regular}. Note that if a
regular point $t$ is represented by a sequence $\ldots x_2x_1$, then $t$ is an interior point of
$\tile_{x_n\ldots x_2x_1}$ for all $n$. Since each tile $\tile_{x_n\ldots x_2x_1}$ is homeomorphic to
$\tile$, the cut-points of $\tile$ also provide information about its local cut-points.

\subsubsection{Components in the tile with a point removed}\label{section_component}

In what follows, let $c(\tile\setminus t)$ denote the number of connected components in~$\tile\setminus t$
for a point $t\in\tile$ and $bc(\tile\setminus t)$ be the number of components in $\tile\setminus t$ that
contain a boundary point of~$\tile$. We will show that the numbers $c(\tile\setminus t)$ and
$bc(\tile\setminus t)$ can be computed using a method similar to the one developed in
Section~\ref{Section_Ends} to find the number of components in the tile graphs with a vertex removed.

%For a word $v\in X^{*}$ let $\tc(v)$ be the number of connected components in $\tile\setminus
%\textrm{int}(\tile_v)$ that contain a boundary point of $\tile_v$.

We start with the following result which is an analog of Propositions~\ref{prop_infcomp_limit}.

\begin{prop}\label{prop_tile_regular_point}
Let $t\in\tile$ be a regular point represented by a sequence $\ldots x_2x_1\in\xmo$. Then
\[
bc(\tile\setminus t)=\lim_{n\rightarrow\infty} bc(\tile\setminus \textrm{int}(\tile_{x_n\ldots x_2x_1}))=
\lim_{n\rightarrow\infty} \pc(T_n\setminus x_n\ldots x_2x_1).
\]
\end{prop}
\begin{proof}
The interior $int(\tile_v)$ of the tile $\tile_v$ is the complement to the subset of finitely many points
that also belong to other tiles of the same level. Therefore $\tile\setminus \textrm{int}(\tile_{v})$ is
the union of all tiles $\tile_u$ for $u\in X^{|v|}, u\neq v$.

Since the point $t$ is regular, we can choose $n$ large enough so that the tile $\tile_{x_n\ldots x_2x_1}$
does not contain the boundary points of $\tile$ contained in $\tile\setminus t$, and every tile $\tile_v$
for $v\in X^n$ contains at most one boundary point of $\tile$. Since $t$ belongs to the interior
$int(\tile_{x_n\ldots x_2x_1})$ of the tile $\tile_{x_n\ldots x_2x_1}$, if two boundary points of $\tile$
lie in the same connected component of $\tile\setminus int(\tile_{x_n\ldots x_2x_1})$, they lie in the
same connected component of $\tile\setminus t$. Therefore the value of the first limit is not less than
$bc(\tile\setminus t)$. Conversely, if two boundary points of $\tile$ lie in the same component of
$\tile\setminus t$, then for sufficiently large $n$ these two points lie in the same components of
$\tile\setminus int(\tile_{x_n\ldots x_2x_1})$. Since the number of boundary points is finite, the first
equality follows.

For the second equality recall that two tiles $\tile_v$ and $\tile_u$ for $v,u\in X^n$ have nonempty
intersection if and only if the vertices $v$ and $u$ are connected by an edge in the graph $T_n$.  It
follows that if two vertices $v$ and $u$ belong to the same component in $T_n\setminus x_n\ldots x_2x_1$,
then the tiles $\tile_v$ and $\tile_u$ belong to the same component in $\tile\setminus t$. Therefore the
value of the second limit is not less than $bc(\tile\setminus t)$. Conversely, since the point $t$ is
regular, one can choose $n$ large enough so that for any pair of boundary points of $\tile$ that belong to
the same connected component in $\tile\setminus t$, these points also belong to the same component in
$\tile\setminus \tile_{x_n\ldots x_2x_1}$. The second equality follows.
\end{proof}

Propositions~\ref{prop_infcomp_limit} and \ref{prop_tile_regular_point} establish the connection between
the number of components in a punctured tile and the number of infinite components in a punctured tile
graph. To describe the limit in Proposition~\ref{prop_infcomp_limit} we constructed the automaton
$\A_{ic}$, which returns the number $\pc(T_n\setminus x_n\ldots x_2x_1)$ by reading the word $x_n\ldots
x_2x_1$ from left to right, so that we can apply it to right-infinite sequences. Similarly one can
construct a finite automaton $\B_{bc}$, which returns $\pc(T_n\setminus x_n\ldots x_2x_1)$ by reading the
word $x_n\ldots x_2x_1$ from right to left (the reversion of a regular language is a regular language) so
that we can apply it to left-infinite sequences. Then we can describe the limit in
Proposition~\ref{prop_tile_regular_point} in the same way as Proposition~\ref{prop_infcomp_A_ic} describes
the limit in Proposition~\ref{prop_infcomp_limit}.

Also we can construct a finite deterministic (acceptor) automaton $\B_{ic}$ with the following property.
The states of $\B_{ic}$ are labeled by tuples of the form
$$(\{\pcset_i\}_i,\{\mathscr{F}_j\}_j,\varphi:\{\pcset_i\}_i\rightarrow\{\mathscr{F}_j\}_j, n).$$ This
tuple indicates the following. For any word $v\in X^{*}$ let us consider the partition of the set of
boundary points of $\tile$ induced by the components of $\tile\setminus int(\tile_v)$ and let $\pcset_i$
be the set of all post-critical sequences representing the points from the same component. Let $n$ be the
number of component in $\tile\setminus int(\tile_v)$ without a boundary point of $\tile$. Similarly, we
consider the boundary points of $\tile_v$ and let $\mathscr{F}_j$ be the set of all post-critical
sequences $p_i$ such that the points represents by $p_iv$ belong to the same component in $\tile\setminus
int(\tile_v)$. Further, we define $\varphi(\pcset_i)=\mathscr{F}_i$ if the points represented by sequences
from $\pcset_i$ and $\mathscr{F}_i$ belong to the same connected component. Then the final state of
$\B_{ic}$ after accepting $v$ is labeled exactly by the constructed tuple. If we are interested just in
the number of all components in $\tile\setminus int(\tile_v)$, we replace the label of each state by the
number of components $\pcset_i$ plus $n$. The following statement follows.

\begin{prop}
For every integer $k$ the set of all words $v\in X^{*}$ with the property that $\tile\setminus
\textrm{int} (\tile_v)$ has $k$ connected components is a regular language recognized by the automaton
$\B_{ic}$.
\end{prop}

This result enables us to provide, in analogy to Theorem~\ref{thm_number_of_ends}, a constructive method
which, given a representation $\ldots x_2x_1$ of a point $t\in\tile$, determines the number of components
in $\tile\setminus t$. In particular, we get a description of cut-points of $\tile$. We distinguish two
cases.

\textit{Regular points.}  If a sequence $\ldots x_2x_1\in\xmo$ represents a regular point $t\in\tile$,
then $t$ is an interior point of the tile $\tile_{x_n\ldots x_2x_1}$ for all $n\geq 1$. Hence, for a
regular point $t$, every connected component of $\tile\setminus t$ intersects the tile $\tile_{x_n\ldots
x_2x_1}$. Since the number of boundary points of each tile $\tile_v$ is not greater than $|\pcset|$, it
follows that $c(\tile\setminus t)$ coincides with the number of components in the partition of the
boundary of the tile $\tile_{x_n\ldots x_2x_1}$ in $\tile\setminus t$ for large enough $n$ (in particular,
$c(\tile\setminus t)\leq |\pcset|$). The last problem can be subdivided on two subproblems. First,
``outside'' subproblem: find how the boundary of the tile $\tile_{x_n\ldots x_2x_1}$ decomposes in
$\tile\setminus int(\tile_{x_n\ldots x_2x_1})$. The automaton $\B_{ic}$ provides an answer to this problem
when we trace the second label $\{\mathscr{F}_j\}_j$ of states. Second, ``inside'' subproblem: find how
the boundary $\partial\tile_{x_n\ldots x_2x_1}$ decomposes in $\tile_{x_n\ldots x_2x_1}\setminus t$. Since
each tile $\tile_v$ is homeomorphic to the tile $\tile$ under the shift map, the automaton $\B_{ic}$
provides an answer to this problem when we trace the first label $\{\pcset_i\}_i$ of states (as well as
the automaton $\B_{bc}$). By combining the two partitions of $\partial\tile_{x_n\ldots x_2x_1}$ given by
the two subproblems, we get the number of connected components in~$\tile\setminus t$.

\textit{Critical points.} Let $t$ be a critical point. Suppose $t$ is represented by a post-critical
sequence $p_1x_1\in\pcset$ with periodic $p_1\in\pcset$ and $x_1\in X$. Note that the structure of bounded
automata implies that periodic post-critical sequences represent regular points of the tile. Therefore we
can apply the previous case to find the number $c(\tile\setminus \overline{p_1})$ and the partition of the
boundary of $\tile$ in $\tile\setminus \overline{p_1}$, where $\overline{p_1}$ stands for the point
represented by $p_1$.

Now consider the components of $\tile\setminus t$. The only difference with the regular case is that $t$ may be not
an interior point of $\tile_v$ for all sufficiently long words $v$. However, it is an interior point of a
union of finitely many tiles, and we will be able to apply the same arguments as above. Consider vertices
$(p_i,x_i), i=2,\ldots,k$ of the model graph $M$ adjacent to the vertex $(p_1,x_1)$. Notice that since
$p_1$ is periodic, all $p_i$ are periodic. The sequences $p_ix_i$ are precisely all sequences that
represent the point $t$. Hence the point $t$ is an interior point of the set $U=\cup_{i=1}^k \tile_{x_i}$.
We can find how the boundary of the set $U$ decomposes in $\tile\setminus \textrm{int}\, U$. Since every
tile $\tile_{x_i}$ is homeomorphic to the tile $\tile$ via the shift map, we can find
$c(\tile_{x_i}\setminus t)=c(\tile\setminus \overline{p_i})$, and deduce the decomposition of the boundary of $\tile_{x_i}$ in $\tile_{x_i}\setminus t$. Combining these partitions, we can find $c(\tile\setminus t)$ and the
corresponding decomposition of $\partial\tile$ (in particular, $c(\tile\setminus t)\leq k|\pcset|$). In
this way we can find $c(\tile\setminus t)$ for every point $t$ represented by a post-critical sequence.

Now let $t$ be a critical point represented by a sequence $pyu$ with $p\in\pcset$, $y\in X$, $u\in X^{*}$,
and $py\not\in\pcset$. Recall that there is a finite automaton, which reads the word $u$ and returns the
decomposition of the boundary of $\tile_u$ in $\tile\setminus \textrm{int}\,\tile_u$. Using the fact that
$\tile_u\setminus t$ and $\tile\setminus \overline{py}$ are homeomorphic, we can find the number
$c(\tile_u\setminus t)$ in the same way as we did above for the sequence $p_1x_1$. Since the
post-critical set is finite, one can construct a finite automaton, which given a finite word $v$ and a
post-critical sequence $p\in\pcset$ returns the number~$c(\tile\setminus \overline{pv})$.

\subsubsection{The number of components in punctured limit space and tile almost surely}

We can use the results from Sections~\ref{section_ends_a_s} and \ref{section_two_ends_a_s} to get information about cut-points of the limit space and tile up to measure zero.

\begin{theor}\label{thm_punctured_tile_ends}
\begin{enumerate}
  \item The number of connected components in $\tile\setminus t$  is the same for almost
all points $t$, and is equal to one or two. Moreover, $c(\tile\setminus t)=2$ almost surely if and only if the Schreier graphs $\Gamma_w$ (equivalently, the tile graphs $T_w$) have two ends for almost all
$w\in\xo$, and in this case the tile $\tile$ is homeomorphic to an interval.

  \item The number of connected components in $\lims\setminus t$ is the same for almost
all points $t$, and is equal to one or two. Moreover, $c(\lims\setminus t)=2$ almost surely if and only if the nucleus $\nucl$ of the group satisfies Case 2 of Theorem~\ref{th_classification_two_ends}.

\end{enumerate}
\end{theor}
\begin{proof}
By Corollary~\ref{cor_one_or_two_ends_a_s} we have to consider only two cases.

If the tile graphs $T_w$ have one end for almost all $w\in\xo$, then by
Remark~\ref{rem_subword_and_inf_comp} there exists a word $v\in X^{*}$ such that $\pc(T_n\setminus
u_1vu_2)=1$ for all $u_1,u_2\in X^{*}$ with $n=|u_1vu_2|$. The set of all sequences of the form $u_1vu_2$
for $u_1\in\xmo$ and $u_2\in X^{*}$ is of full measure. Then by Proposition~\ref{prop_tile_regular_point} we
get that all boundary points of the tile $\tile$ belong to the same component in $\tile\setminus t$ for
almost all points $t$. Since every tile $\tile_v$ is homeomorphic to $\tile$, we get that the boundary of
every tile $\tile_v$ belongs to the same component in $\tile_v\setminus t$ for almost all points $t$. It
follows that $c(\tile\setminus t)=1$ almost surely. Since the limit space $\lims$ can be constructed from $\tile$ by gluing finitely many of its points, in this case we get $c(\lims\setminus t)=1$ almost surely.

If the tile graphs $T_w$ have two ends almost surely, then we are in
the settings of Theorem~\ref{th_classification_two_ends}. It is direct to check that in both cases of this theorem the tile $\tile$ is homeomorphic to an interval (because the tile graphs $T_n$ are intervals), and hence $c(\tile\setminus t)=2$ for almost all points $t$.
In Case 1 the limit space is homeomorphic to a circle and therefore $c(\lims\setminus t)=1$ almost surely.
In Case 2 the limit space is homeomorphic to an interval and therefore $c(\lims\setminus t)=2$ almost surely.
\end{proof}

%Now we can apply the characterization result from Theorem~\ref{th_classification_two_ends} to limit
%spaces.

\begin{cor}\label{cor_limsp_interval_circle}
\begin{enumerate}
\item The tile $\tile$ of a contracting self-similar group with open set condition is homeomorphic to an interval if and
only if the nucleus of the group satisfies Theorem~\ref{th_classification_two_ends}.
\item The limit space $\lims$ of a contracting self-similar group with connected tiles and open set condition is
homeomorphic to a circle if and only if the nucleus of the group satisfies Case 1 of
Theorem~\ref{th_classification_two_ends}, i.e., it consists of the adding machine, its
inverse, and the identity state.

\item The limit space $\lims$ of a contracting self-similar group with connected tiles and open set condition is
homeomorphic to an interval if and only if the nucleus of the group satisfies Case 2 of
Theorem~\ref{th_classification_two_ends}.
\end{enumerate}
\end{cor}
\begin{proof}
Let $G$ be a contracting self-similar groups with open set condition, and let the tile $\tile$ of the
group $G$ be homeomorphic to an interval. Then the group $G$ has connected tiles and all tiles $\tile_v$
are homeomorphic to an interval. The boundary of tiles is finite, hence the group $G$ is generated by a
bounded automaton (here we use the open set condition), and we are under the settings of this section.
Since $c(\tile\setminus t)=2$ almost surely, the Schreier graphs $\Gamma_w$ with respect to the nucleus $\nucl$ have almost surely two ends, and hence $\nucl$ satisfies Theorem~\ref{th_classification_two_ends}.

It is left to prove the statements about limit spaces. A small connected neighborhood of any point of a
circle or of an interval is homeomorphic to an interval. Hence, if the limit space $\lims$ is a circle or an interval, the tile $\tile$ is homeomorphic to an
interval. Therefore we are in the settings of
Theorem~\ref{th_classification_two_ends}. As was mentioned above, in Case~1 of
Theorem~\ref{th_classification_two_ends} the limit space $\lims$ is homeomorphic to a circle, and in
Case~2 it is homeomorphic to an interval.
\end{proof}

The last corollary together with Theorem~\ref{th_Sch_2ended_binary} agree with the following result of
Nekrashevych and \v{S}uni\'{c}:

\begin{theor}[{\cite[Theorem~5.5]{img}}]
The limit dynamical system $(\lims,\si)$ of a contracting self-similar group $G$ is topologically conjugate to the tent map if and only if $G$ is equivalent as a self-similar group to one of the automata $\A_{\omega, \rho}$.
\end{theor}

\section{Examples}\label{Section_Examples}

\subsection{Basilica group}
The Basilica group $G$ is generated by the automaton shown in Figure~\ref{fig_BasilicaAutomaton}. This
group is the iterated monodromy group of $z^2-1$. It is torsion-free, has exponential growth, and is the
first example of amenable but not subexponentially amenable group (see \cite{gri_zuk:spect_pro}). The
orbital Schreier graphs $\gr_w$ of this group have polynomial growth of degree $2$ (see
\cite[Chapter~VI]{PhDBondarenko}). The structure of Schreier graphs $\gr_w$ was investigated in
\cite{ddmn:GraphsBasilica}. In particular, it was shown that there are uncountably many pairwise
non-isomorphic graphs $\gr_w$ and the number of ends was described. Let us show how to get the result
about ends using the developed method.

\begin{figure}
\begin{center}
\psfrag{00}{$0|0$} \psfrag{11}{$1|1$} \psfrag{10}{$1|0$} \psfrag{01}{$0|1$} \psfrag{id}{$id$}
\psfrag{a}{\large$a$} \psfrag{b}{\large$b$} \psfrag{c}{\large$c$} \psfrag{a0}{\small$a0$}
\psfrag{b0}{\small$b0$} \psfrag{c0}{\small$c0$} \psfrag{a1}{\small$a1$} \psfrag{b1}{\small$b1$}
\psfrag{c1}{\small$c1$}
\epsfig{file=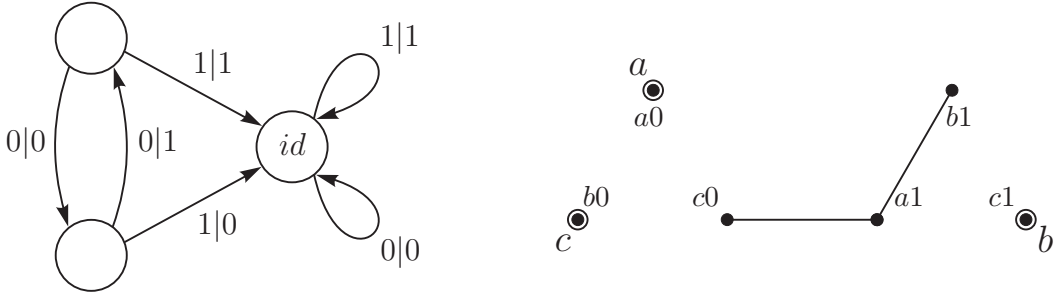,height=110pt}
\caption{Basilica automaton and its
model graph}\label{fig_BasilicaAutomaton}
\end{center}
\end{figure}

\begin{figure}
\begin{center}
\psfrag{00}{$00$} \psfrag{11}{$11$} \psfrag{10}{$10$} \psfrag{01}{$01$} \psfrag{0}{$0$} \psfrag{1}{$1$}
\psfrag{b|c}{$b|c$} \psfrag{ac|b}{$ac|b$} \psfrag{ab|c}{$ab|c$} \psfrag{ab}{$ab$} \psfrag{ac}{$ac$}
\psfrag{abc}{$abc$} \psfrag{ac|b|1}{\small$ac|b|1$} \psfrag{ab|c|1}{\small$ab|c|1$} \psfrag{ab|1}{$ab|1$}
\psfrag{ac|1}{$ac|1$} \psfrag{abc|1}{$abc|1$} \psfrag{AC}{\LARGE $\A_{c}$} \psfrag{AIC}{\LARGE $\A_{ic}$}
\epsfig{file=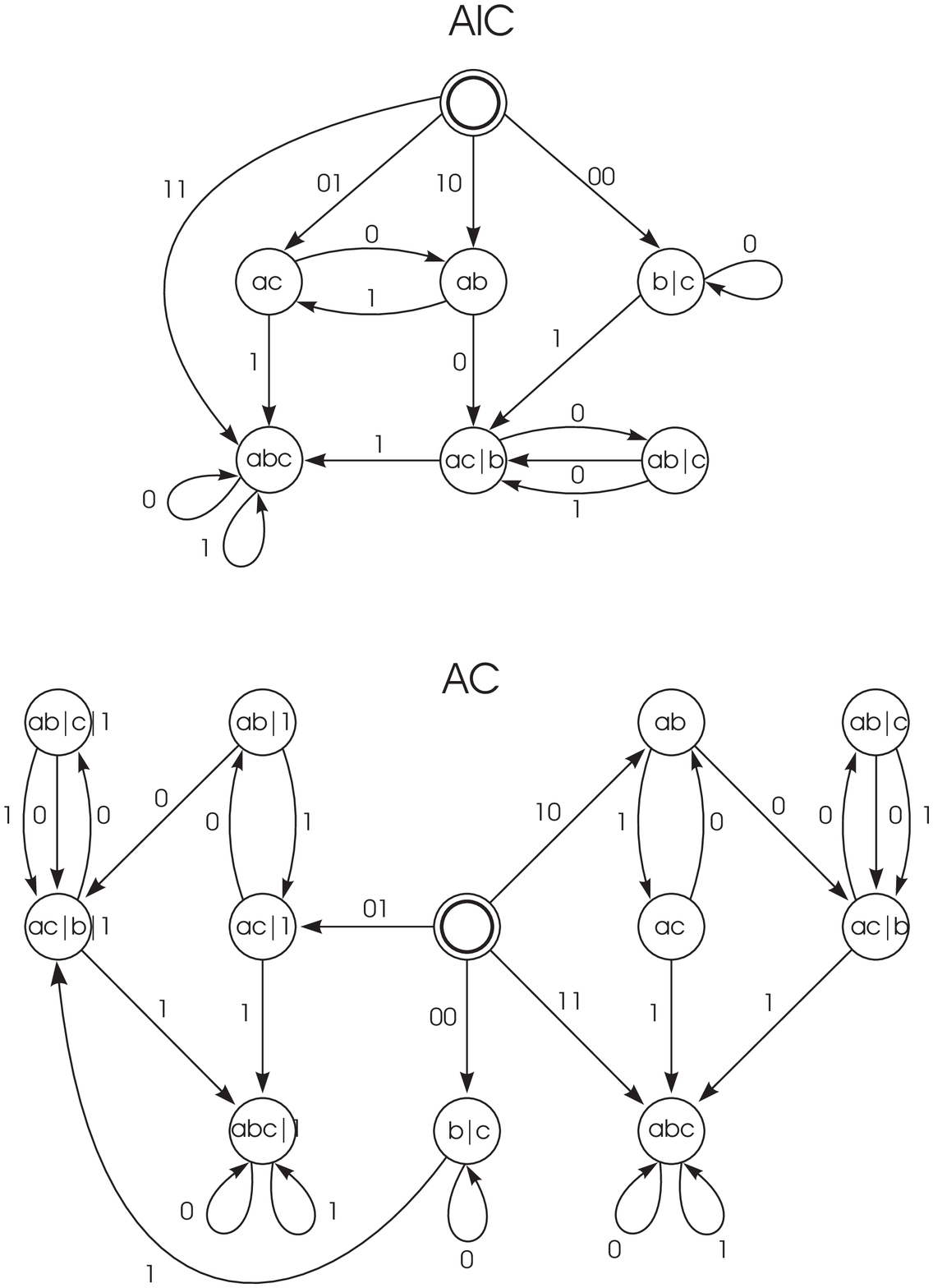,height=550pt}
\caption{The automata $\A_{ic}$ and $\A_{c}$ for Basilica
group}
\label{fig_Basilica_Ac_Aic}
\end{center}
\end{figure}

The alphabet is $X=\{0,1\}$ and the post-critical set $\pcset$ consists of three elements $a=0^{-\omega}$,
$b=(01)^{-\omega}$, $c=(10)^{-\omega}$. The model graph is shown in Figure~\ref{fig_BasilicaAutomaton}.
The automata $\A_{c}$ and $\A_{ic}$ are shown in Figure~\ref{fig_Basilica_Ac_Aic}. We get that each tile
graph $T_w$ has one or two ends, and we denote by $E_1$ and $E_2$ the corresponding sets of sequences $w$.
For the critical sequences $w=0^{\omega}$ the tile graph $T_w$ has two ends, while for the other critical
sequences $(01)^{\omega}$ and $(10)^{\omega}$ the tile graph $T_w$ has one end. Using the automaton $\A_{ic}$
the sets $E_1$ and $E_2$ can be described by Theorem~\ref{thm_number_of_ends} as follows:
\begin{eqnarray*}
E_2&=& X^{*}(0X)^{\omega}\setminus\left( Cof((01)^{\omega}\cup (10)^{\omega}) \right),\\
E_1&=&\xo\setminus E_2.
\end{eqnarray*}
Almost every tile graph $T_w$ has one end, the set $E_2$ is uncountable but of measure zero.

%E_2&=&\left\{ x_1x_2\ldots x_n 0 x_{n+1} 0 x_{n+2} \ldots | n\geq 0 \mbox{ and } x_i\in X \right\}\setminus
%\left( Cof((01)^{\omega}\cup (10)^{\omega}) \right),\\

Every graph $T_w\setminus w$ has one, two, or three connected components, and we denote by $C_1$, $C_2$,
and $C_3$ the corresponding sets of sequences. Using the automaton $\A_c$ these sets can be described precisely
as follows:
\begin{eqnarray*}
C_3&=&\bigcup_{k\geq 0} \left( 010(10)^k0 (0X)^{\omega} \cup 000^k1 (0X)^{\omega} \right),\\
C_2&=&\bigcup_{k\geq 1} (10)^k0(0X)^{\omega} \bigcup \left(00\xo \cup 01\xo\right) \setminus C_3, \\
C_1&=&\xo\setminus \left( C_2\cup C_3\right).
\end{eqnarray*}
The set $C_3$ is uncountable but of measure zero, while the sets $C_1$ and $C_2$ are of measure~$1/2$.

%C_3&=&\left\{010(10)^k0 0 x_1 0 x_2 \ldots, 000^k10 x_1 0 x_2 \ldots | k\geq 0 \mbox{ and } x_i\in X\right\},\\
%C_2&=& \left( 00\xo \cup 01\xo \cup \left\{ (10)^k00x_10x_2\ldots | k\geq 1 \mbox{ and } x_i\in X \right\}\right) \setminus C_3, \\

Each graph $T_w\setminus w$ has one or two infinite components. The corresponding sets $IC_1$ and $IC_2$
can be described using the automaton $\A_{ic}$ as follows:
\begin{eqnarray*}
IC_2&=&\bigcup_{k\geq 1} \left( (10)^k 0 (0X)^{\omega} \cup 0(10)^k0 (0X)^{\omega} \cup 00^k1 (0X)^{\omega}
\right)
\setminus \left(Cof((01)^{\omega}\cup (10)^{\omega}) \right),\\
IC_1&=& \xo\setminus IC_2.
\end{eqnarray*}
The set $IC_2$ is uncountable but of measure zero.

%IC_2&=&\left\{ (10)^k 00x_10x_2\ldots, 0(10)^k00x_10x_2\ldots, 00^k10x_10x_2\ldots | k\geq 1 \mbox{ and } x_i\in
%X\right\} \setminus\\ & & \setminus
%\left( Cof((01)^{\omega}\cup (10)^{\omega}) \right),\\

\begin{figure}
\begin{center}
\psfrag{00000}{\tiny$00000$} \psfrag{10000}{\tiny$10000$} \psfrag{01000}{\tiny$01000$}
\psfrag{11000}{\tiny$11000$} \psfrag{00100}{\tiny$00100$} \psfrag{10100}{\tiny$10100$}
\psfrag{01100}{\tiny$01100$} \psfrag{11100}{\tiny$11100$} \psfrag{00010}{\tiny$00010$}
\psfrag{10010}{\tiny$10010$} \psfrag{01010}{\tiny$01010$} \psfrag{11010}{\tiny$11010$}
\psfrag{00110}{\tiny$00110$} \psfrag{10110}{\tiny$10110$} \psfrag{01110}{\tiny$01110$}
\psfrag{11110}{\tiny$11110$} \psfrag{00001}{\tiny$00001$} \psfrag{10001}{\tiny$10001$}
\psfrag{01001}{\tiny$01001$} \psfrag{11001}{\tiny$11001$} \psfrag{00101}{\tiny$00101$}
\psfrag{10101}{\tiny$10101$} \psfrag{01101}{\tiny$01101$} \psfrag{11101}{\tiny$11101$}
\psfrag{00011}{\tiny$00011$} \psfrag{10011}{\tiny$10011$} \psfrag{01011}{\tiny$01011$}
\psfrag{11011}{\tiny$11011$} \psfrag{00111}{\tiny$00111$} \psfrag{10111}{\tiny$10111$}
\psfrag{01111}{\tiny$01111$} \psfrag{11111}{\tiny$11111$} \epsfig{file=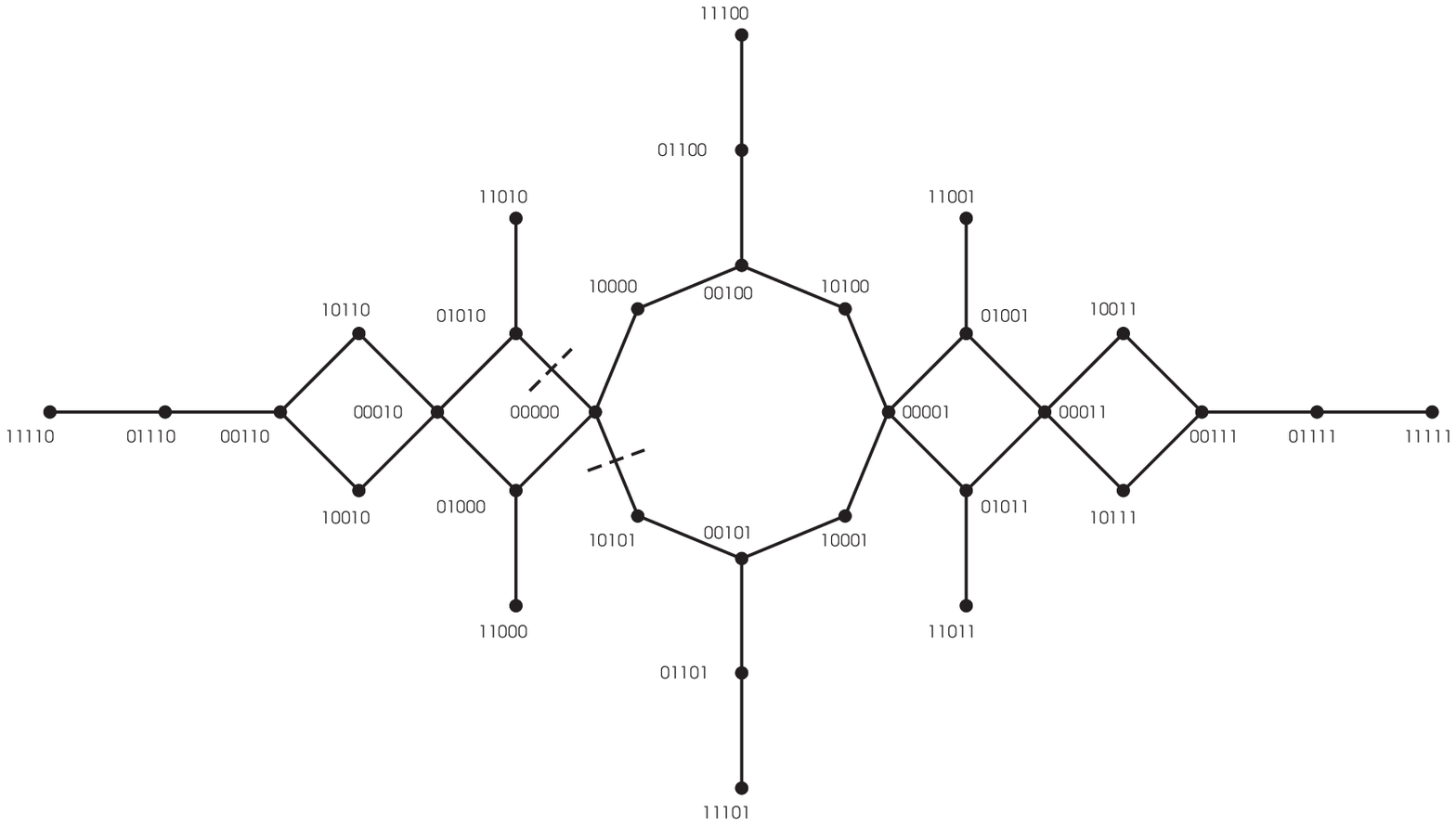,height=230pt}
\epsfig{file=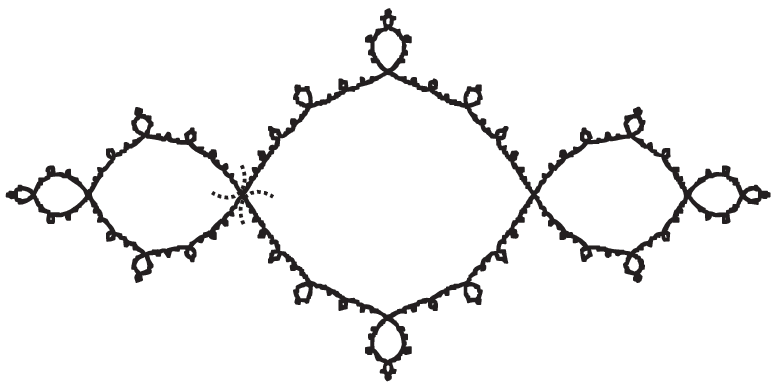,height=160pt}
\caption{The Schreier graph $\gr_{5}$ of the Basilica
group and its limit space}\label{fig_Basilica_Graph LimitSpace}
\end{center}
\end{figure}

The finite Schreier graph $\gr_n$ differs from the finite tile graph $T_n$ by two edges $\{a_n,b_n\}$ and
$\{a_n,c_n\}$. Assuming these edges one can relabel the states of the automaton $\A_c$ so that it returns
the number of components in $\gr_n\setminus v$. In this way we get that $c(\gr_n\setminus v)=1$ if the
word $v$ starts with $10$ or $11$; in the other cases $c(\gr_n\setminus v)=2$. In particular, the Schreier
graph $\gr_n$ has $2^{n-1}$ cut-vertices.

The orbital Schreier graph $\gr_w$ coincides with the the tile graph $T_w$ except when $w$ is critical. The
critical sequences $0^{\omega}$, $(01)^{\omega}$, and $(10)^{\omega}$ lie in the same orbit and the
corresponding Schreier graph consists of three tile graphs $T_{0^{\omega}}$, $T_{(01)^{\omega}}$,
$T_{(10)^{\omega}}$ with two new edges $(0^{\omega},(01)^{\omega})$ and $(0^{\omega},(10)^{\omega})$. It
follows that this graph has four ends.

The limit space $\lims$ of the group $G$ is homeomorphic to the Julia set of $z^2-1$ shown in
Figure~\ref{fig_Basilica_Graph LimitSpace}. The tile $\tile$ can be obtained from the limit space by
cutting the limit space in the way shown in the figure, or, vise versa, the limit space can be obtained
from the tile by gluing points represented by post-critical sequences $0^{-\omega}$, $(01)^{-\omega}$,
$(10)^{-\omega}$. Every point $t\in\tile$ divides the tile $\tile$ into one, two, or three connected
components. Put $\mathscr{C}=\{0^{-\omega}1, (01)^{-\omega}1, (10)^{-\omega}0\}$. Then the sets
$\mathscr{C}_1$, $\mathscr{C}_2$, and $\mathscr{C}_3$ of sequences from $\xmo$, which represent the
corresponding cut-points, can be described as follows:
\begin{eqnarray*}
\mathscr{C}_3&=& \bigcup_{n\geq 0} \mathscr{C}(0X)^n \cup \mathscr{C}(0X)^n0,\\
\mathscr{C}_2&=& \bigcup_{n\geq 0} \left(\mathscr{C}(X0)^n \cup \mathscr{C}(X0)^nX\right)\bigcup \left(
(0X)^{-\omega}\cup (X0)^{-\omega}\right)\setminus \left(\mathscr{C}_3\cup \left\{(10)^{-\omega}, (01)^{-\omega}\right\}\right),\\
\mathscr{C}_1&=&\xmo\setminus \left(\mathscr{C}_2\cup\mathscr{C}_3\right).
\end{eqnarray*}
The set $\mathscr{C}_3$ of three-section points is countable, the set $\mathscr{C}_2$ of bisection points
is uncountable and of measure zero, and the tile $\tile\setminus t$ is connected for almost all points $t$.

Every point $t\in\lims$ divides the limit space $\lims$ into one or two connected components. The
corresponding sets $\mathscr{C}'_1$ and $\mathscr{C}'_2$ can be described as follows:
\begin{eqnarray*}
\mathscr{C}'_2&=& \bigcup_{n\geq 0} \left(\mathscr{C}(X0)^n \cup \mathscr{C}(0X)^n \cup \mathscr{C}(0X)^n0
\cup \mathscr{C}(X0)^nX\right)\bigcup\\ && \bigcup \left(
(0X)^{-\omega}\cup (X0)^{-\omega}\right)\setminus \left\{(10)^{-\omega}, (01)^{-\omega}\right\},\\
\mathscr{C}'_1&=&\xmo\setminus \mathscr{C}'_2.
\end{eqnarray*}
The set $\mathscr{C}'_2$ of bisection points is uncountable and of measure zero, and the limit space
$\lims\setminus t$ is connected for almost all points $t$.

\begin{figure}
\begin{center}
\psfrag{00}{$0|0$} \psfrag{11}{$1|1$} \psfrag{10}{$1|0$} \psfrag{01}{$0|1$} \psfrag{id}{$id$}
\psfrag{22}{$2|2$} \psfrag{20}{$2|0$} \psfrag{12}{$1|2$} \psfrag{a}{\large$a$} \psfrag{b}{\large$b$}
\psfrag{a0}{\small$a0$} \psfrag{b0}{\small$b0$} \psfrag{a1}{\small$a1$} \psfrag{b1}{\small$b1$}
\psfrag{a2}{\small$a2$}
\psfrag{b2}{\small$b2$}
\epsfig{file=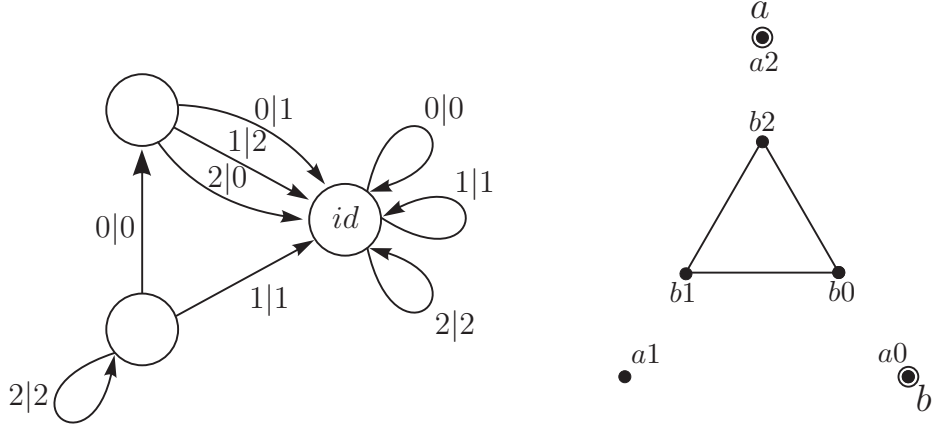,height=160pt}
\caption{Gupta-Fabrykowski
automaton and its model graph}\label{fig_GuptaFabrAutomaton}
\end{center}
\end{figure}

\subsection{Gupta-Fabrykowski group}
The Gupta-Fabrykowski group $G$ is generated by the automaton shown in Figure~\ref{fig_GuptaFabrAutomaton}.
It was constructed in \cite{gupta_fabr2} as an example of a group of intermediate growth. Also this group
is the iterated monodromy group of $z^3(-\frac{3}{2}+i\frac{\sqrt{3}}{2})+1$ (see
\cite[Example~6.12.4]{self_sim_groups}). The Schreier graphs $\gr_w$ of this group were studied in
\cite{barth_gri:spectr_Hecke}, where their spectrum and growth were computed (they have polynomial growth
of degree $\frac{\log 3}{\log 2}$).

The alphabet is $X=\{0,1,2\}$ and the post-critical set $\pcset$ consists of two elements $a=2^{-\omega}$
and $b=2^{-\omega}0$. The model graph is shown in Figure~\ref{fig_GuptaFabrAutomaton}. The automata
$\A_{c}$ and $\A_{ic}$ are shown in Figure~\ref{fig_GuptaFabr_Ac_Aic}. Every Schreier graph $\gr_w$
coincides with the tile graph $T_w$. We get that every tile graph $T_w$ has one or two ends, and we denote
by $E_1$ and $E_2$ the corresponding sets of sequences. For the only critical sequence $2^{\omega}$ the
tile graph $T_w$ has one end. Using the automaton $\A_{ic}$ the sets $E_1$ and $E_2$ can be described by
Theorem~\ref{thm_number_of_ends} as follows:
\begin{eqnarray*}
E_2=X^{*} \{0,2\}^{\omega}\setminus Cof(2^{\omega}),\qquad E_1=\xo\setminus E_2.
\end{eqnarray*}
Almost every tile graph has one end, the set $E_2$ is uncountable but of measure zero.

\begin{figure}
\begin{center}
\psfrag{0}{$0$} \psfrag{1}{$1$} \psfrag{a}{$a$} \psfrag{b}{$b$} \psfrag{a|b}{$a|b$} \psfrag{ab}{$ab$}
\psfrag{ab|1}{$ab|1$} \psfrag{AC}{\LARGE $\A_{c}$} \psfrag{AIC}{\LARGE $\A_{ic}$}

\epsfig{file=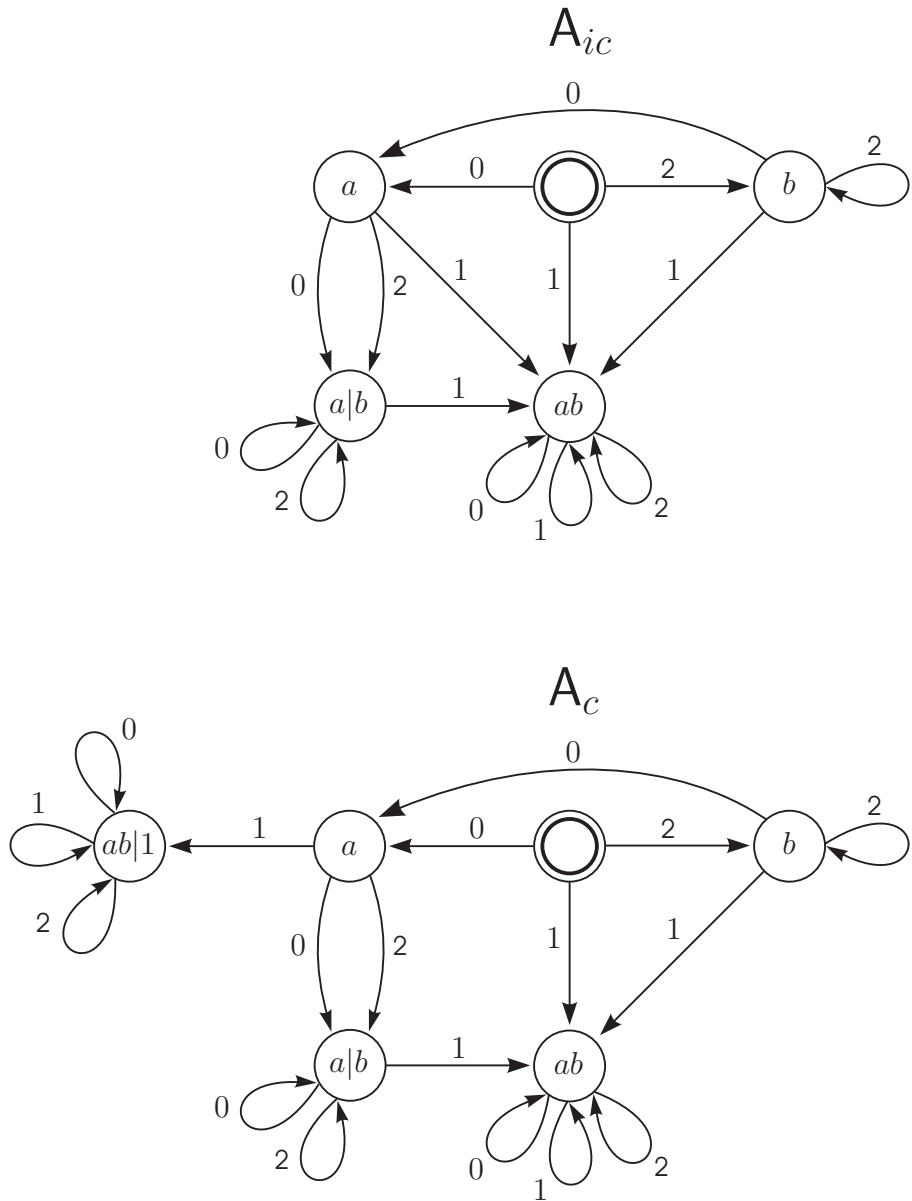,height=450pt}
\caption{The automata $\A_{ic}$ and $\A_{c}$ for
Gupta-Fabrykowski group}\label{fig_GuptaFabr_Ac_Aic}
\end{center}
\end{figure}

Every graph $T_w\setminus w$ has one or two connected components, and we denote by $C_1$ and $C_2$ the
corresponding sets of sequences. Using the automaton $\A_c$ these sets can be described precisely as follows:
\begin{eqnarray*}
C_2= \bigcup_{k\geq 0} \left(2^k01\xo\cup 2^k0\{0,2\}^{\omega}\right), \qquad C_1=\xo\setminus C_2.
\end{eqnarray*}
The sets $C_1$ and $C_2$ have measure $\frac{5}{6}$ and $\frac{1}{6}$ respectively.

Every graph $T_w\setminus w$ has one or two infinite components. The corresponding sets $IC_1$ and $IC_2$
can be described using the automaton $\A_{ic}$ as follows:
\begin{eqnarray*}
IC_2=\bigcup_{k\geq 0} 2^k0\{0,2\}^{\omega} \setminus Cof(2^{\omega}),\qquad  IC_1=\xo\setminus IC_2.
\end{eqnarray*}
The set $IC_2$ is uncountable but of measure zero.

The limit space $\lims$ and the tile $\tile$ of the group $G$ are homeomorphic to the Julia set of the map
$z^3(-\frac{3}{2}+i\frac{\sqrt{3}}{2})+1$ shown in Figure~\ref{fig_GuptaFabr_Graph LimitSpace}. Every
point $t\in\lims$ divides the limit space into one, two, or three connected components. The sets
$\mathscr{C}_1$, $\mathscr{C}_2$, and $\mathscr{C}_3$ of sequences from $\xmo$, which represent the
corresponding points, can be described as follows:
\begin{eqnarray*}
\mathscr{C}_3&=&2^{-\omega}0X^{*}\setminus \{2^{-\omega}0\},\\
\mathscr{C}_2&=&\{0,2\}^{-\omega}\setminus \left(\mathscr{C}_3\cup \{2^{-\omega}, 2^{-\omega}0\}\right),\\
\mathscr{C}_1&=&\xmo\setminus \left(\mathscr{C}_2\cup\mathscr{C}_3\right).
\end{eqnarray*}
The set $\mathscr{C}_3$ of three-section points is countable, the set $\mathscr{C}_2$ of bisection points
is uncountable and of measure zero, and the limit space $\lims\setminus t$ is connected for almost all
points $t$.

\begin{figure}
\begin{center}
\psfrag{000}{\tiny$000$} \psfrag{100}{\tiny$100$} \psfrag{200}{\tiny$200$} \psfrag{010}{\tiny$010$}
\psfrag{110}{\tiny$110$} \psfrag{210}{\tiny$210$} \psfrag{020}{\tiny$020$} \psfrag{120}{\tiny$120$}
\psfrag{220}{\tiny$220$} \psfrag{001}{\tiny$001$} \psfrag{101}{\tiny$101$} \psfrag{201}{\tiny$201$}
\psfrag{011}{\tiny$011$} \psfrag{111}{\tiny$111$} \psfrag{211}{\tiny$211$} \psfrag{021}{\tiny$021$}
\psfrag{121}{\tiny$121$} \psfrag{221}{\tiny$221$} \psfrag{002}{\tiny$002$} \psfrag{102}{\tiny$102$}
\psfrag{202}{\tiny$202$} \psfrag{012}{\tiny$012$} \psfrag{112}{\tiny$112$} \psfrag{212}{\tiny$212$}
\psfrag{022}{\tiny$022$} \psfrag{122}{\tiny$122$} \psfrag{222}{\tiny$222$}
\epsfig{file=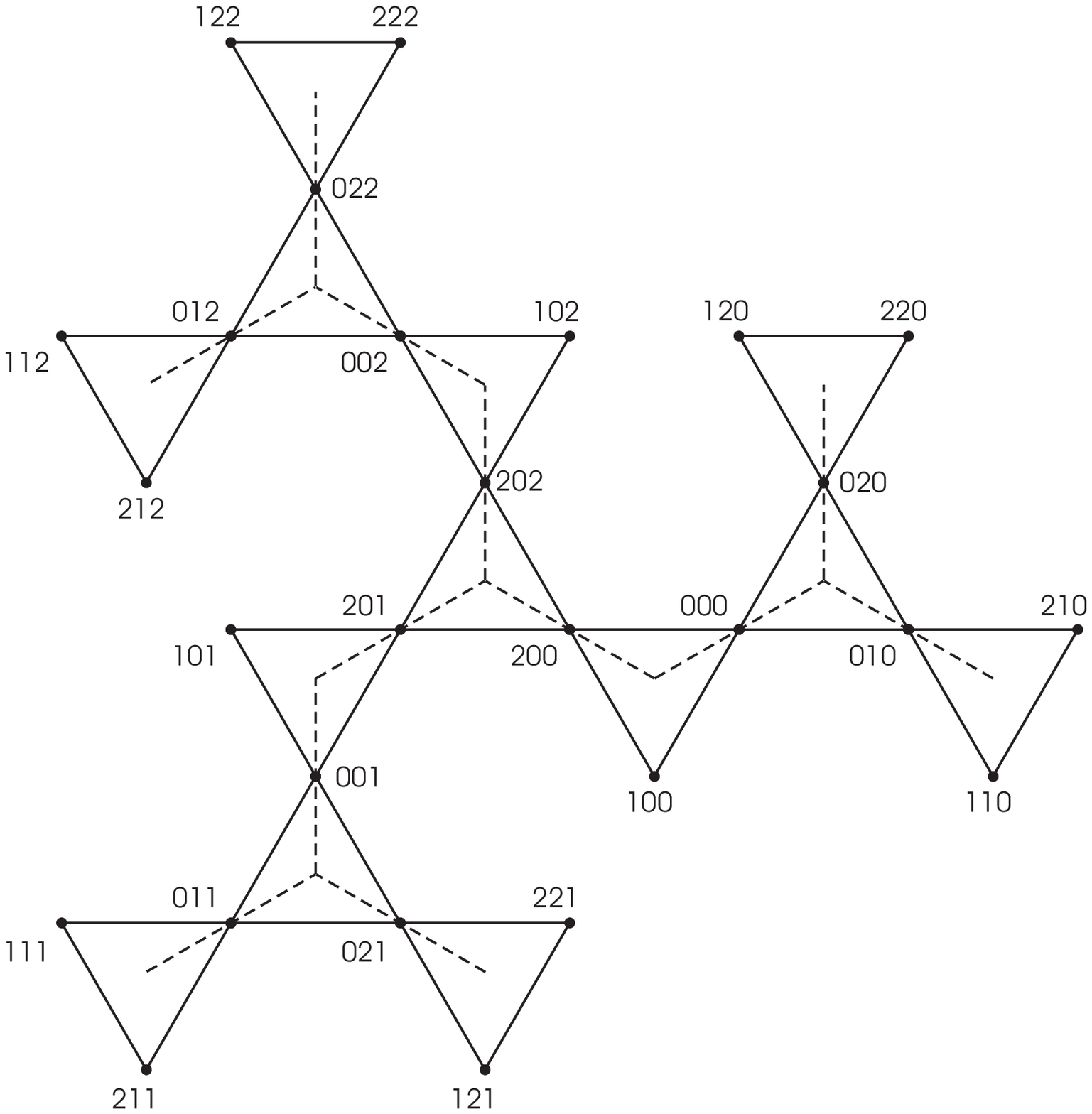,height=200pt} \epsfig{file=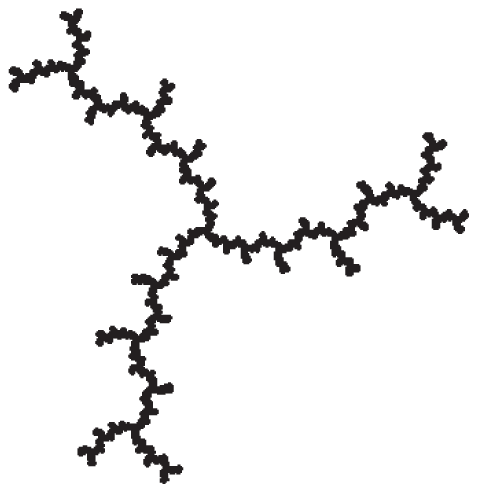,height=200pt}
\caption{The
Schreier graph $\gr_{3}$ of the Gupta-Fabrykowski group and its limit space}\label{fig_GuptaFabr_Graph
LimitSpace}
\end{center}
\end{figure}

%Every critical point $x\in\lims$ is represented by three sequences $2^{-\omega}00u$, $2^{-\omega}01u$, and
%$2^{-\omega}02u$ for some $u\in X^{*}$.

\subsection{Iterated monodromy group of $z^2+i$}

The iterated monodromy group of $z^2+i$ is generated by the automaton shown in
Figure~\ref{fig_IMGz2iAutomaton}. This group is one more example of a group of intermediate growth (see
\cite{bux_perez:img}). The algebraic properties of $IMG(z^2+i)$ were studied in \cite{img_z2_i}. The
Schreier graphs $\gr_w$ of this group have polynomial growth of degree $\frac{\log 2}{\log \lambda}$, where
$\lambda$ is the real root of $x^3-x-2$ (see \cite[Chapter~VI]{PhDBondarenko}).

\begin{figure}
\begin{center}
\psfrag{00}{$0|0$} \psfrag{11}{$1|1$} \psfrag{10}{$1|0$} \psfrag{01}{$0|1$} \psfrag{id}{$id$}
\psfrag{a}{\large$a$} \psfrag{b}{\large$b$} \psfrag{c}{\large$c$} \psfrag{a0}{\small$a0$}
\psfrag{b0}{\small$b0$} \psfrag{c0}{\small$c0$} \psfrag{a1}{\small$a1$} \psfrag{b1}{\small$b1$}
\psfrag{c1}{\small$c1$}\epsfig{file=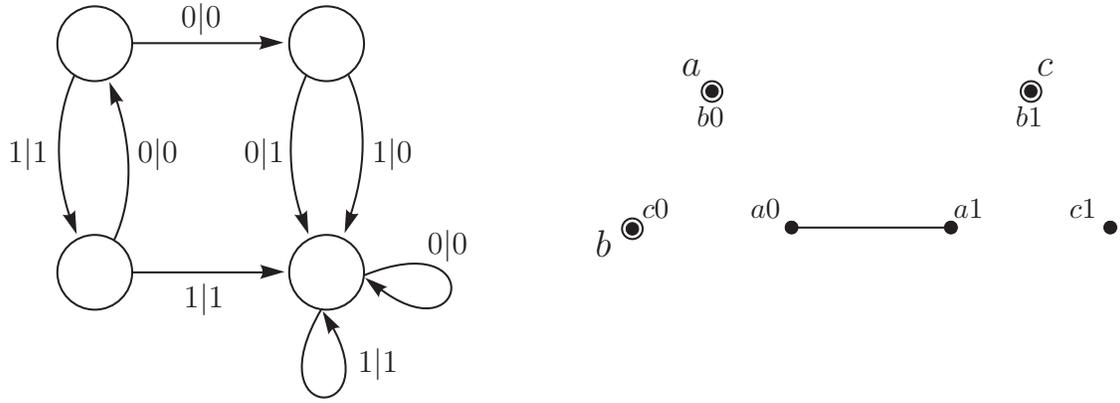,height=150pt}
\caption{$IMG(z^2+i)$ automaton and
its model graph}\label{fig_IMGz2iAutomaton}
\end{center}
\end{figure}

The alphabet is $X=\{0,1\}$ and the post-critical set $\pcset$ consists of three elements
$a=(10)^{-\omega}0$, $b=(10)^{-\omega}$, and $c=(01)^{-\omega}$. The model graph is shown in
Figure~\ref{fig_IMGz2iAutomaton}. The automata $\A_{c}$ and $\A_{ic}$ are shown in
Figure~\ref{fig_IMGz2i_Aic}. Every Schreier graph $\gr_w$ coincides with the tile graph $T_w$ and it is a
tree. We get that every tile graph $T_w$ has one, two, or three ends, and we denote by $E_1$, $E_2$, and
$E_3$ the corresponding sets of sequences. Using the automaton $\A_{ic}$ the sets $E_1$, $E_2$, $E_3$ can be
described by Theorem~\ref{thm_number_of_ends} as follows. For the both critical sequences
$(10)^{\omega}$ and $(01)^{\omega}$ the tile graph $T_w$ has one end. Denote by $\mathscr{R}$ the right one-sided
sofic subshift given by the subgraph emphasized in Figure~\ref{fig_IMGz2i_Aic}. Then
\begin{eqnarray*}
E_3= Cof(0^{\omega}), \ \ E_2=X^{*}\mathscr{R} \setminus Cof(0^{\omega}\cup (10)^{\omega}\cup
(01)^{\omega}), \ \ E_1=\xo\setminus E_2.
\end{eqnarray*}
(We cannot describe these sets in
the way we did with the previous examples, because the subshift $\mathscr{R}$ is not of finite type).
Almost every tile graph has one end, the set $E_2$ is uncountable but of measure zero, and there is one
graph, namely $T_{0^{\omega}}$, with three ends. This example shows that
Corollary~\ref{cor_>2ends_pre_periodic} may hold for regular sequences (here $0^{\omega}$ is regular).

\begin{figure}
\begin{center}
\psfrag{00}{$00$} \psfrag{11}{$11$} \psfrag{10}{$10$} \psfrag{01}{$01$} \psfrag{0}{$0$} \psfrag{1}{$1$}
\psfrag{b|c}{$b|c$} \psfrag{bc}{$bc$} \psfrag{a|bc}{$a|bc$} \psfrag{ac|b}{$ac|b$} \psfrag{ab|c}{$ab|c$}
\psfrag{ab}{$ab$} \psfrag{ac}{$ac$} \psfrag{abc}{$abc$} \psfrag{a|b|c}{$a|b|c$} \psfrag{AC}{\LARGE
$\A_{c}$} \psfrag{AIC}{\LARGE $\A_{ic}$} \epsfig{file=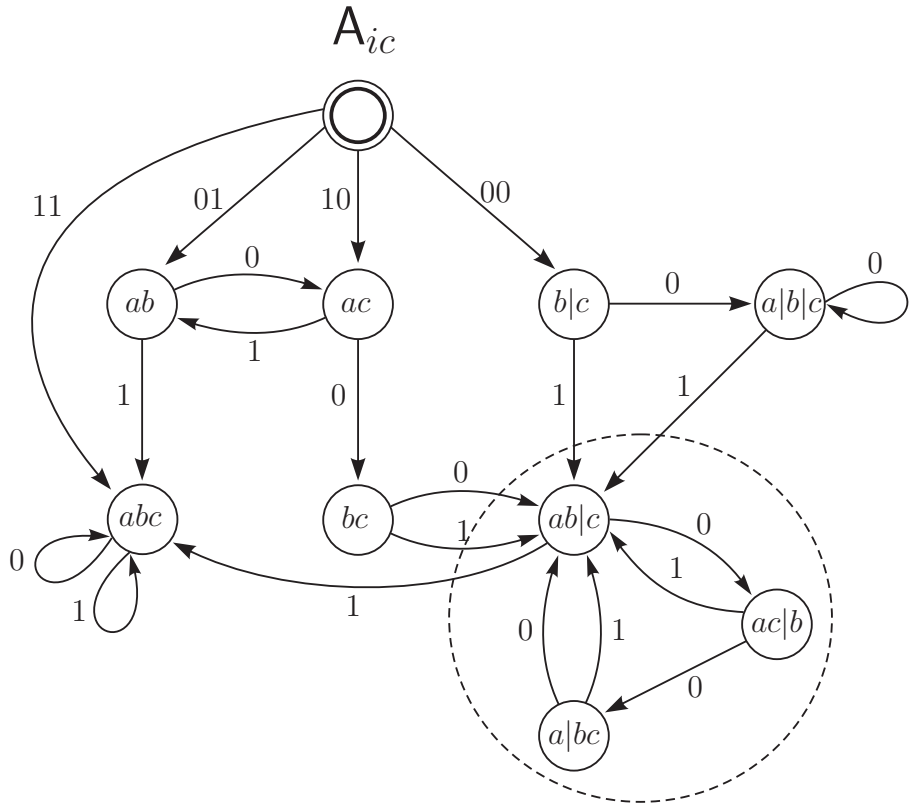,height=300pt}
\caption{The automaton
$\A_{ic}$ for $IMG(z^2+i)$}\label{fig_IMGz2i_Aic}
\end{center}
\end{figure}

\begin{figure}
\begin{center}
\psfrag{00}{$00$} \psfrag{11}{$11$} \psfrag{10}{$10$} \psfrag{01}{$01$} \psfrag{0}{$0$} \psfrag{1}{$1$}
\psfrag{b|c}{$b|c$} \psfrag{bc}{$bc$} \psfrag{bc|1}{$bc|1$}\psfrag{a|bc}{$a|bc$} \psfrag{ac|b}{$ac|b$}
\psfrag{ab|c}{$ab|c$} \psfrag{ab}{$ab$} \psfrag{ac}{$ac$} \psfrag{abc}{$abc$} \psfrag{a|b|c}{$a|b|c$}
\psfrag{ab|1}{$ab|1$} \psfrag{ac|1}{$ac|1$} \psfrag{abc|1}{$abc|1$} \psfrag{a|bc|1}{\small$a|bc|1$}
\psfrag{ac|b|1}{\small$ac|b|1$} \psfrag{ab|c|1}{\small$ab|c|1$} \psfrag{AC}{\LARGE $\A_{c}$}
\psfrag{AIC}{\LARGE $\A_{ic}$} \epsfig{file=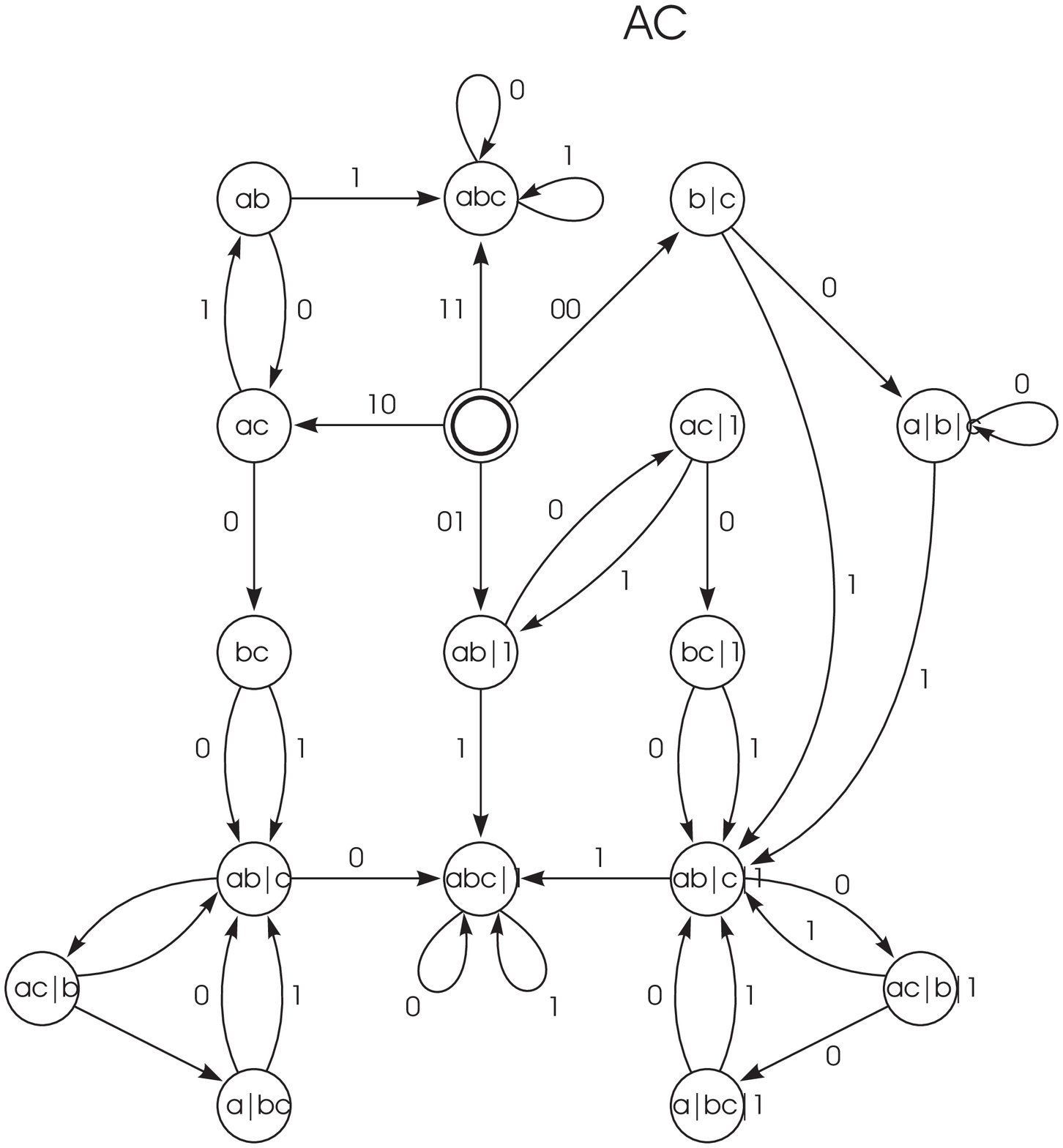,height=450pt}
\caption{The automaton $\A_{c}$ for
$IMG(z^2+i)$}\label{fig_IMGz2i_Ac}
\end{center}
\end{figure}

Every graph $T_w\setminus w$ has one, two, or three connected components, and we denote by $C_1$, $C_2$,
and $C_3$ the corresponding sets of sequences. Using the automaton $\A_c$ these sets can be described precisely
as follows:
\begin{eqnarray*}
C_3=\bigcup_{k\geq 0} 0(10)^k0X\mathscr{R}\bigcup_{k\geq 2} 0^k1\mathscr{R}\bigcup \{0^{\omega}\},\quad
C_2=\xo\setminus \left(C_3\cup C_1\right),\quad C_1=\bigcup_{k\geq 0} 1(01)^k1\xo.
\end{eqnarray*}
The set $C_3$ is of measure zero, and the sets $C_1$ and $C_2$ have measure $\frac{1}{3}$ and $\frac{2}{3}$
respectively.

Every graph $T_w\setminus w$ has one, two, or three infinite components. The corresponding sets $IC_1$ and
$IC_2$ can be described using the automaton $\A_{ic}$ as follows:
\begin{eqnarray*}
IC_2&=&\bigcup_{k\geq 1} \left(0^k01\mathscr{R}\cup (10)^k0X\mathscr{R}\cup 0(10)^k0X\mathscr{R}\right),\\
IC_3&=&\{0^{\omega}\},\quad IC_1=\xo\setminus \left(IC_2\cup IC_3\right).
\end{eqnarray*}
The set $IC_2$ is uncountable but of measure zero.

\begin{figure}
\begin{center}
\psfrag{00000}{\tiny$00000$} \psfrag{10000}{\tiny$10000$} \psfrag{01000}{\tiny$01000$}
\psfrag{11000}{\tiny$11000$} \psfrag{00100}{\tiny$00100$} \psfrag{10100}{\tiny$10100$}
\psfrag{01100}{\tiny$01100$} \psfrag{11100}{\tiny$11100$} \psfrag{00010}{\tiny$00010$}
\psfrag{10010}{\tiny$10010$} \psfrag{01010}{\tiny$01010$} \psfrag{11010}{\tiny$11010$}
\psfrag{00110}{\tiny$00110$} \psfrag{10110}{\tiny$10110$} \psfrag{01110}{\tiny$01110$}
\psfrag{11110}{\tiny$11110$} \psfrag{00001}{\tiny$00001$} \psfrag{10001}{\tiny$10001$}
\psfrag{01001}{\tiny$01001$} \psfrag{11001}{\tiny$11001$} \psfrag{00101}{\tiny$00101$}
\psfrag{10101}{\tiny$10101$} \psfrag{01101}{\tiny$01101$} \psfrag{11101}{\tiny$11101$}
\psfrag{00011}{\tiny$00011$} \psfrag{10011}{\tiny$10011$} \psfrag{01011}{\tiny$01011$}
\psfrag{11011}{\tiny$11011$} \psfrag{00111}{\tiny$00111$} \psfrag{10111}{\tiny$10111$}
\psfrag{01111}{\tiny$01111$} \psfrag{11111}{\tiny$11111$} \epsfig{file=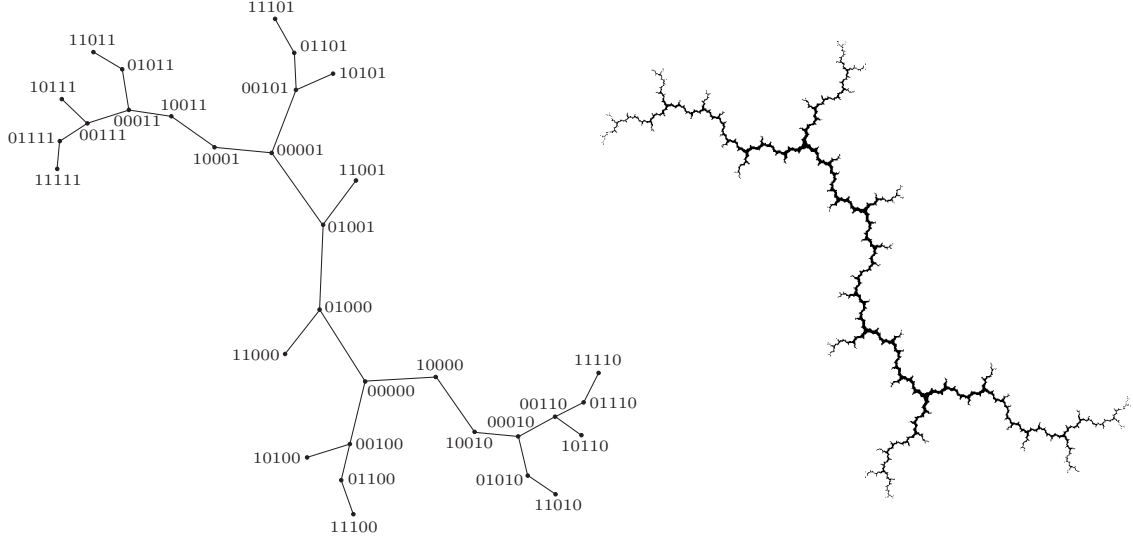,height=200pt}
\caption{The Schreier graph $\gr_{5}$ of the IMG($z^2+i$) and its limit space}\label{fig_IMG_Graph
LimitSpace}
\end{center}
\end{figure}

The limit space $\lims[]$ and the tile $\tile$ of the group $IMG(z^2+i)$ are homeomorphic to the Julia set
of the map $z^2+i$ shown in Figure~\ref{fig_IMG_Graph LimitSpace}. Every point $t\in\lims[]$ divides the
limit space into one, two, or three connected components. The sets $\mathscr{C}_1$, $\mathscr{C}_2$, and
$\mathscr{C}_3$ of sequences from $\xmo$, which represent the corresponding points, can be described as
follows:
\begin{eqnarray*}
\mathscr{C}_3= Cof(0^{-\omega}),\quad \mathscr{C}_2= \mathscr{L} X^{*}\setminus \mathscr{C}_3,\quad
\mathscr{C}_1= \xmo\setminus \left(\mathscr{C}_2\cup \mathscr{C}_3\right),
\end{eqnarray*}
where $\mathscr{L}$ is the left one-sided sofic subshift given by the subgraph emphasized in
Figure~\ref{fig_IMGz2i_Aic}. The set $\mathscr{C}_3$ of three-section points is countable, the set
$\mathscr{C}_2$ of bisection points is uncountable and of measure zero, and the limit space
$\lims[]\setminus t$ is connected for almost all points $t$.


\begin{thebibliography}{10}

\bibitem{AL}
D.~Aldous and R.~Lyons,
\newblock Processes on unimodular random networks,
\newblock {\em Electr. J. Probab.} \textbf{12} (2007), 1454--1508.

\bibitem{barth_gri:spectr_Hecke}
L.~Bartholdi and R.~Grigorchuk,
\newblock On the spectrum of {H}ecke type operators related to some fractal groups,
\newblock {\em Tr. Mat. Inst. Steklova} \textbf{231} (2000), 5--45.

%(Din. Sist., Avtom. i Beskon.  Gruppy):

\bibitem{fractal_gr_sets}
L.~Bartholdi, R.~Grigorchuk, and V.~Nekrashevych,
\newblock From fractal groups to fractal sets,
\newblock in {\em Fractals in Graz 2001}, Trends Math., Birkh\"auser, Basel, 2003, 25--118.

\bibitem{bhn:aut_til}
L.~Bartholdi, A.~G.~Henriques, and V.~Nekrashevych,
\newblock {Automata, groups, limit spaces, and tilings},
\newblock {\em J. Algebra}  \textbf{305} (2006), 629--663.

%\bibitem{on_dyn_vert}
%A.~Blokh and G.~Levin.
%\newblock {On dynamics of vertices of locally connected polynomial Julia sets.}
%\newblock {\em Proc. Am. Math. Soc.}, 130(11):3219--3230, 2002.


\bibitem{bondnek:pcf}
I.~Bondarenko and V.~Nekrashevych,
\newblock Post-critically finite self-similar groups,
\newblock {\em Algebra Discrete Math.} \textbf{4} (2003), 21--32.


\bibitem{PhDBondarenko}
I.~Bondarenko,
\newblock Groups generated by bounded automata and their {S}chreier graphs.
\newblock Ph.D. Dissertation, Texas A\&M University, 2007.

\bibitem{GrowthSch}
I.~Bondarenko,
\newblock {Growth of {S}chreier graphs of automaton groups,}
\newblock {\em Math. Ann.} \textbf{354} (2012), 765--785.

\bibitem{ZigZag:bond}
I.~Bondarenko,
\newblock {Self-similar groups and the zig-zag and replacement products of graphs,}
\newblock {\em Journal of Algebra} \textbf{434} (2015), 1--11.

\bibitem{bk:meas_limsp}
I.~Bondarenko and R.~Kravchenko,
\newblock {Graph-directed systems and self-similar measures on limit spaces of
  self-similar groups,}
\newblock {\em Adv. Math.} \textbf{226} (2011), 2169--2191.


\bibitem{bux_perez:img}
K.-U.~Bux and R.~P{\'e}rez,
\newblock On the growth of iterated monodromy groups,
\newblock in: {\em Topological and asymptotic aspects of group theory}, vol.~394,
Contemp. Math., Amer. Math. Soc., Providence, RI, 2006, 61--76.

\bibitem{ddmn:GraphsBasilica}
D.~D'Angeli, A.~Donno, M.~Matter, and T.~Nagnibeda,
\newblock {Schreier graphs of the {B}asilica group,}
\newblock {\em J. Mod. Dyn.}  \textbf{4} (2010), 167--205.

\bibitem{ddn:Ising}
D.~D'Angeli, A.~Donno, and T.~Nagnibeda,
\newblock {Partition functions of the {I}sing model on some self-similar {S}chreier graphs,}
\newblock in: {\em Random Walks, Boundaries and Spectra} (D. Lenz, F. Sobieczky and W. Woess Eds.),
Progress in Prob. Vol. 64, Birkh\"auser, Springer, Basel, 2011, 277--304.

\bibitem{ddn:Dimer}
D.~D'Angeli, A.~Donno, and T.~Nagnibeda,
\newblock {Counting dimer coverings on self-similar Schreier graphs,}
\newblock {\em  European Journal of Combinatorics} \textbf{33} (2012), 1484--1513.

\bibitem{delzant_grigorchuk}
T.~Delzant and R.~Grigorchuk,
\newblock Homomorphic images of branch groups, and Serre's property (FA),
\newblock in: {\em Geometry and dynamics of groups and spaces}, Progr. Math., 265, Birkh\"{a}user, Basel, 2008,
353--375.

\bibitem{gupta_fabr2}
J.~Fabrykowski and N.~Gupta,
\newblock On groups with sub-exponential growth functions {II},
\newblock {\em J. Indian Math. Soc. (N.S.)} \textbf{56} (1991), 217--228.

\bibitem{GNS}
R.~Grigorchuk, V.~Nekrashevych, and V.~Sushchanskii,
\newblock Automata, dynamical systems and groups,
\newblock {\em Tr. Mat. Inst. Steklova} \textbf{231} (2000), 134--214.

\bibitem{img_z2_i}
R.~Grigorchuk, D.~Savchuk, and Z.~{\v{S}}uni\'c,
\newblock {The spectral problem, substitutions and iterated monodromy,}
\newblock {\em CRM Proceedings and Lecture Notes} \textbf{42} (2007), 225--248.


\bibitem{gri_sunik:hanoi}
R.~Grigorchuk and Z.~{\v{S}}uni{\'c},
\newblock{Asymptotic aspects of {S}chreier graphs and {H}anoi {T}owers groups},
\newblock {\em C. R. Math. Acad. Sci. Paris} \textbf{342} (2006), 545--550.


\bibitem{gri_zuk:lampl_group}
R.~Grigorchuk and A.~{\.Z}uk,
\newblock {The lamplighter group as a group generated by a 2-state automaton, and its spectrum},
\newblock {\em Geom. Dedicata}  \textbf{87} (2001), 209--244.


\bibitem{gri_zuk:spect_pro}
R.~Grigorchuk and A.~{\.Z}uk,
\newblock {Spectral properties of a torsion-free weakly branch group defined by a three state automaton,}
\newblock in: {\em Computational and statistical group theory (Las Vegas, NV/Hoboken, NJ, 2001)},
vol.~298, Contemp. Math., Amer. Math. Soc., Providence, RI, 2002, 57--82.

\bibitem{gromov}
M.~Gromov,
\newblock{Structures m\'{e}triques pour les vari\'{e}t\'{e}s riemanniennes},
\newblock{Textes Math\'{e}matiques, J. La-fontaine and P. Pansu (Eds.)},
1. CEDIC, Paris, 1981. iv+152 pp. ISBN: 2-7124-0714-8.

\bibitem{kigami:anal_fract}
J.~Kigami,
\newblock {\em Analysis on fractals},
\newblock Volume 143 of Cambridge Tracts in Mathematics, University Press, Cambridge, 2001.


\bibitem{MN_AMS}
M.~Matter and T.~Nagnibeda,
\newblock {Abelian sandpile model on randomly rooted graphs and self-similar groups,}
\newblock {\em Israel J. Math.} \textbf{199} (2014), 363--420.

\bibitem{Previte}
J.~P.~Previte,
\newblock {Graph substitutions,}
\newblock {\em Ergodic Theory Dynam. Systems} \textbf{18} (1998), 661--685.

\bibitem{self_sim_groups}
V.~Nekrashevych,
\newblock {\em Self-similar groups},
\newblock Volume 117 of {\em Mathematical Surveys
  and Monographs}, Amer. Math. Soc., Providence, RI, 2005.

\bibitem{ssgroups_geom}
V.~Nekrashevych,
\newblock {Self-similar groups and their geometry,}
\newblock {\em Sa\~{o} Paulo Journal of Mathematical Sciences}  \textbf{1} (2007), 41--96.


\bibitem{img}
V.~Nekrashevych,
\newblock {Iterated monodromy groups,}
\newblock in: {\em Groups St Andrews 2009 in Bath}, vol.~1, London Math. Soc.
Lecture Note Ser. 387, Cambridge University Press, Cambridge, 2011, 41--93.


\bibitem{Zoran}
Z.~\v{S}uni\'{c},
\newblock {Hausdorff dimension in a family of self-similar groups,}
\newblock {\em Geometriae Dedicata} \textbf{124} (2007), 213--236.

\bibitem{sidki:circ}
S.~Sidki,
\newblock {Automorphisms of one-rooted trees: growth, circuit structure, and acyclicity,}
\newblock {\em J. Math. Sci. (New York)} \textbf{100} (2000), 1925--1943.

\bibitem{smirnov}
S. K. Smirnov,
\newblock {On supports of dynamical laminations and biaccessible points in polynomial Julia sets,}
\newblock \textit{Colloq. Math.} \textbf{87} (2001), 287--295.

\bibitem{zdunik}
A. Zdunik,
\newblock {On biaccessible points in Julia sets of polynomials,}
\newblock \textit{Fund. Math.} \textbf{163} (2000), 277--286.

\end{thebibliography}
\end{document}